\documentclass[sn-mathphys,Numbered]{sn-jnl}

\usepackage{graphicx}%
\usepackage{multirow}%
\usepackage{amsmath,amssymb,amsfonts}%
\usepackage{amsthm}%
\usepackage{mathrsfs}%
\usepackage[title]{appendix}%
\usepackage{xcolor}%
\usepackage{textcomp}%
\usepackage{manyfoot}%
\usepackage{booktabs}%
\usepackage{algorithm}%
\usepackage{algorithmicx}%
\usepackage{algpseudocode}%
\usepackage{listings}%

\theoremstyle{thmstyleone}%
\newtheorem{theorem}{Theorem}
\newtheorem{proposition}[theorem]{Proposition}%

\theoremstyle{thmstyletwo}%
\newtheorem{remark}{Remark}%

\theoremstyle{thmstylethree}%
\newtheorem{definition}{Definition}%

\newtheorem{lemma}{Lemma}
\newtheorem{corollary}{Corollary}

\newcommand{\bn}{{\mathbf n}}

\newcommand{\bu}{{\mathbf u}}

\newcommand{\bv}{{\mathbf v}}
\newcommand{\bw}{{\mathbf w}}

\newcommand{\bff}{{\mathbf f}}
\newcommand{\bA}{{\mathbf A}}
\newcommand{\bB}{{\mathbf B}}

\newcommand{\bD}{{\mathbf D}}

\newcommand{\bF}{{\mathbf F}}
\newcommand{\bG}{{\mathbf G}}

\newcommand{\bI}{{\mathbf I}}

\newcommand{\bK}{{\mathbf K}}
\newcommand{\bL}{{\mathbf L}}
\newcommand{\bM}{{\mathbf M}}
\newcommand{\bN}{{\mathbf N}}
\newcommand{\bV}{{\mathbf V}}
\newcommand{\bS}{{\mathbf S}}
\newcommand{\bU}{{\mathbf U}}

\newcommand{\bP}{{\mathbf \Phi}}
\newcommand{\DV}{{\rm Div}\,}
\newcommand{\dv}{{\rm div}\,}

\newcommand{\BR}{{\mathbb R}}
\newcommand{\BC}{{\mathbb C}}

\newcommand{\BN}{{\mathbb N}}

\newcommand{\CA}{{\mathcal A}}
\newcommand{\CB}{{\mathcal B}}
\newcommand{\CC}{{\mathcal C}}
\newcommand{\CD}{{\mathcal D}}
\newcommand{\CE}{{\mathcal E}}
\newcommand{\CF}{{\mathcal F}}

\newcommand{\CL}{{\mathcal L}}
\newcommand{\CM}{{\mathcal M}}
\newcommand{\CN}{{\mathcal N}}
\newcommand{\CG}{{\mathcal G}}
\newcommand{\CR}{{\mathcal R}}
\newcommand{\CS}{{\mathcal S}}
\newcommand{\CT}{{\mathcal T}}

\newcommand{\CX}{{\mathcal X}}
\newcommand{\CY}{{\mathcal Y}}
\newcommand{\CZ}{{\mathcal Z}}
\newcommand{\R}{{\mathbb R}}
\newcommand{\C}{{\mathbb C}}
\newcommand{\N}{{\mathbb N}}

\newcommand{\fA}{{\mathfrak A}}
\newcommand{\fr}{{\mathfrak r}}
\newcommand{\fB}{{\mathfrak B}}
\newcommand{\fC}{{\mathfrak C}}

\newcommand{\fp}{{\mathfrak p}}
\newcommand{\fq}{{\mathfrak q}}
\newcommand{\fm}{{\mathfrak m}}
\newcommand{\fa}{{\mathfrak a}}

\newcommand{\bg}{{\mathbf g}}

\newcommand{\pd}{\partial}

\renewcommand{\appendixname}{Appendix}

\newcommand{\fl}{\mathfrak l}
\renewcommand{\Re}{\operatorname{Re}}
\renewcommand{\Im}{\operatorname{Im}}
\newcommand{\vp}{\varphi}

\newcommand{\X}{{\mathscr X}}

\begin{document}

\title[$\CR$-bounded operator families arising from a compressible fluid model of Korteweg type
with surface tension in the half-space]{$\CR$-bounded operator families arising from a compressible fluid model of Korteweg type
with surface tension in the half-space}


\author*[1]{\fnm{Sri} \sur{Maryani}}\email{sri.maryani@unsoed.ac.id}


\author[2]{\fnm{Miho} \sur{Murata}}\email{murata.miho@shizuoka.ac.jp}
\equalcont{These authors contributed equally to this work.}

\affil*[1]{\orgdiv{Department of Mathematics}, \orgname{Faculty of Mathematics and Natural Sciences, Jenderal Soedirman University}, \orgaddress{\street{Karangwangkal}, \city{Purwokerto}, \postcode{53123}, \state{Central Java}, \country{Indonesia}}}

\affil[2]{\orgdiv{Department of Mathematical and System Engineering}, \orgname{Faculty of Engineering, Shizuoka University}, \orgaddress{\street{3-5-1 Johoku, Naka-ku}, \city{Hamamatsu-shi}, \postcode{432-8561}, \state{Shizuoka}, \country{Japan}}}


\abstract{In this paper, we consider a resolvent problem 
arising from the free boundary value problem for the compressible fluid model 
of Korteweg type, which is called as the {\it Navier-Stokes-Korteweg system}, with surface tension in the half-space.
The Navier-Stokes-Korteweg system is known as a diffuse interface model for liquid-vapor two-phase flows.
Our purpose is to show the $\CR$-boundedness for the solution operator families of the resolvent problem, which gives us the maximal regularity estimates 
in the $L_p$-in-time and $L_q$-in-space setting by applying the Weis's
operator valued Fourier multiplier theorem \cite{W}.}




\maketitle

\section{Introduction}\label{sec1}

\subsection{Model}
Let $\R^N_+$ and $\R^N_0$ be the upper half-space and  its boundary for $N \ge 2$, respectively; namely,
\begin{align*}
	\R^N_+&=\{x=(x', x_N) \mid x' \in \R^{N-1}, x_N>0\},\\
	\R^N_0&=\{x=(x', x_N) \mid x' \in \R^{N-1}, x_N=0\},
\end{align*}
where $x'=(x_1, \ldots, x_{N-1})$.
In this paper, we consider the following resolvent problem 
arising from a compressible fluid model of Korteweg type in the half-space with taking the surface tension in account.
\begin{equation}\label{r1}
\left\{
\begin{aligned}
	&\lambda \rho + \rho_* \dv \bu = d & \quad&\text{in $\R^N_+$}, \\
	&\rho_* \lambda \bu - \DV\{\bS(\bu) - (\gamma_* - \rho_* \kappa \Delta) \rho \bI\}=\bff& \quad&\text{in $\R^N_+$},\\
	&\{\bS(\bu) -(\gamma_* - \rho_* \kappa \Delta) \rho \bI\} \bn - \sigma \Delta' h \bn =\bg & \quad&\text{on $\R^N_0$},\\
	&\bn \cdot \nabla \rho = k & \quad&\text{on $\R^N_0$},\\
	&\lambda h - \bu \cdot \bn = \zeta & \quad&\text{on $\R^N_0$},
\end{aligned}
\right.
\end{equation}
where $\lambda$ is the resolvent parameter varying in the sectorial region
\[
	\Sigma_{\epsilon, \lambda_0} = \{\lambda \in \C \setminus\{0\} \mid |\arg \lambda| < \pi-\epsilon, |\lambda| \ge \lambda_0\}
\]
for $0<\epsilon<\pi/2$ and $\lambda_0 \ge 0$;
$\rho = \rho(x)$, $x=(x_1, \ldots, x_N)$
and  
$\bu = \bu(x) = (u_1(x), \ldots, u_N(x))^\mathsf{T} \footnote{$\bA^\mathsf{T}$denotes the transpose of $\bA$.}$
are unknown density field and velocity field, respectively; $h=h(x)$ is an unknown function on $\R^N_{0}$ obtained by linearization of the kinematic equation, 
$\Delta' h = \sum^{N-1}_{j=1} \pd_j ^2 h$; $\bn=(0, \ldots, 0, -1)^\mathsf{T}$ is the unit outer normal to $\R^N_0$;
$\rho_*, \kappa, \sigma>0$ and $\gamma_*\in \R$ are constants;
the right members $d=d(x)$, $\bff=\bff(x)= (f_1(x), \ldots, f_N(x))^\mathsf{T}$, $\bg=\bg(x)= (g_1(x), \ldots, g_N(x))^\mathsf{T}$, $k=k(x)$, and $\zeta=\zeta(x)$ are given functions.
The viscous stress tensor $\bS(\bu)$ is given by 
\begin{align*}
\bS (\bu) &= \mu \bD(\bu) + 
(\nu - \mu) \dv \bu \bI, 
\end{align*} 
where $\mu>0$ and $\nu>0$ are the viscosity coefficients,
$\bD(\bu)$ denotes 
the deformation tensor whose $(j, k)$ components are
 $D_{jk}(\bu) = \pd_ju_k
+ \pd_ku_j$ with $\pd_j
= \pd/\pd x_j$.
For any vector of functions $\bv = (v_1, \ldots, v_N)^\mathsf{T}$, 
we set $\dv \bv = \sum_{j=1}^N\pd_jv_j$ with $\pd_j=\pd/\pd x_j$,
and also for any $N\times N$ matrix field $\bL$ with $(j,k)^{\rm th}$ components $L_{jk}$, 
the quantity $\DV \bL$ is an 
$N$-vector with $j^{\rm th}$ component $\sum_{k=1}^N\pd_kL_{jk}$;
$\bI=(\delta_{ij})_{1\le i, j \le N}$ is the $N\times N$ identity matrix.

A diffuse interface model for liquid-vapor two-phase flows
was introduced by Korteweg \cite{K} based on Van der Waals's idea \cite{Wa}
and derived rigorously by Dunn and Serrin \cite{DS}.
There are many results for the following whole space problem:
\begin{equation}\label{comp}
\left\{
\begin{aligned}
	&\pd_t \rho + \dv (\rho \bu) = 0 & &\quad\text{in $\R^N$, $t>0$}, \\
	&\rho (\pd_t \bu + \bu \cdot \nabla \bu)  
	= \DV (\bS(\bu)+\bK(\rho) - P(\rho)\bI)& &\quad\text{in $\R^N$, $t>0$}, \\	
	&(\rho, \bu)|_{t=0} = (\rho_0, \bu_0)& &\quad\text{in $\R^N$},
\end{aligned}
\right.
\end{equation}
where $\pd_t = \pd/\pd t$, $t$ is the time variable.
Here  
$\bK(\rho)$ is the 
Korteweg stress tensor given by 
\begin{align*}
\bK (\rho) &= \frac{\kappa}{2} (\Delta \rho^2 -  |\nabla \rho|^2 )\bI 
- \kappa \nabla \rho \otimes \nabla \rho,
\end{align*} 
 where 
 $\kappa$ is the capillary coefficient;
$\nabla \rho \otimes \nabla \rho$ denotes an $N\times N$ matrix with $(j, k)^{\rm th}$
component $(\pd_j\rho)(\pd_k\rho)$
for $\nabla \rho = (\pd_1 \rho, \dots, \pd_N \rho)^\mathsf{T}$;
$P(\rho)$ is the pressure field 
satisfying a $C^\infty$ function defined on
$\rho\ge 0$;
$\rho_0$ and $\bu_0$ are given initial data. 

We refer to mathematical results on system \eqref{comp}.
Bresch, Desjardins, and Lin \cite{BDL} 
proved the existence of global weak solution,
 and then Haspot improved their result in \cite{H2011}.
Hattori and Li \cite{HL1, HL2} first showed
the local and global unique existence of the strong solution
in Sobolev space. 
Hou, Peng, and Zhu \cite{HPZ} improved the results \cite{HL1, HL2}
when the total energy is small. 
Danchin and Desjardins \cite{DD} proved the local and global existence and uniqueness of the solution 
in critical Besov space.
Chikami and Kobayashi \cite{CK} improved the result \cite{DD}.
In particular, in the case $P'(\rho_*) = 0$ for some constant $\rho_*$, 
they proved the global estimates
under an additional low frequency assumption
to control a pressure term.
For the whole space problem in the case $P'(\rho_*)=0$,
we also refer to Kobayashi and Tsuda \cite{KT}; Kobayashi and the second author \cite{KM}.
Large time behavior of solutions was established by 
Wang and Tan \cite{WT}, 
Tan and Wang \cite{TW}, 
Tan, Wang, and Xu \cite{TWX}, and 
Tan and Zhang \cite{TZ} 
in $L_2$-setting; the second author and Shibata \cite{MS} in $L_p$-$L_q$ setting.
For the optimality of decay estimates in the $L^p$ critical framework, we refer to Kawashima, Shibata, and Xu \cite{KSX}.

Regarding boundary value problems,
the system \eqref{comp} was studied in a domain $\Omega$ with the boundary condition:
\begin{equation}\label{bc}
	\bu=0, \quad \bn \cdot \nabla \rho=0 \quad\text{on $\Gamma$.}
\end{equation}
Kotschote \cite{K2008} proved the existence and uniqueness of strong solutions in both bounded and exterior domains locally in time in the $L_p$-setting.
He also considered a non-isothermal case for Newtonian and Non-Newtonian fluids in \cite{K2010, K2012}.
Moreover, he proved the global existence and exponential decay estimates of strong solutions in a bounded domain for small initial data in \cite{K2014}. 
Saito \cite{S2020} proved the existence of $\CR$-bounded solution operator families 
for the large resolvent parameter, and then he obtained the maximal $L_p$-$L_q$ regularity for the linearized problem in uniform $C^3$ domains and a generation of continuous analytic semigroup
associated with the linearized problem.
Concerning the resolvent estimates with small resolvent parameter,
Kobayashi, the second author, and Saito \cite{KMS} proved the resolvent estimate for $\lambda \in \overline{\C_+}=\{z \in \C \mid \Re z\ge 0\}$ in a bounded domain
under the condition that the pressure $P(\rho)$ satisfies not only $P'(\rho_*) \ge 0$
but also $P'(\rho_*)<0$, and then they obtained a global solvability.
Moreover, they proved the resolvent estimate for $\lambda \in \overline{\C_+}$ with $|\lambda| \ge \delta$ for any $\delta>0$ in an exterior domain if $P'(\rho_*) \ge 0$.

On the other hand, we are interested in a free boundary value problem with the surface tension;
namely, 
the domain $\Omega$ and the boundary condition \eqref{bc} are replaced by a time-dependent domain $\Omega_t$ and
\begin{equation}\label{nsk}
\left\{
\begin{aligned}
	&(\bS(\bu)+\bK(\rho)-P(\rho)\bI)\bn_t
	=-P(\rho_*)\bn_t+\sigma(H(\Gamma_t)-H(\Gamma_0))\bn_t& &\quad\text{on $\Gamma_t$, $t>0$}, \\
	&\bn_t \cdot \nabla \rho=0 & &\quad\text{on $\Gamma_t$, $t>0$}, \\	
	&V_{\Gamma_t}=\bu \cdot \bn_t & &\quad\text{on $\Gamma_t$, $t>0$},
\end{aligned}
\right.
\end{equation}
where 
$\bn_t$ is the unit outer normal;
$\sigma$ is the coefficient of the surface tension;
$H(\Gamma_t)$ is the $N-1$-fold mean curvature on $\Gamma_t$;
$\Gamma_0$ is the boundary of the given initial domain $\Omega_0$.
The third equation of \eqref{nsk} is called the kinematic equation, where
$V_{\Gamma_t}$ is the velocity of the evolution of free surface $\Gamma_t$ in the direction of $\bn_t$.
 
Since $\Omega_t$ is unknown, we transform $\Omega_t$ to the fixed domain $\Omega_0$ by the so-called {\it Hanzawa transform} \cite{H}, 
and then the system of the linearized equations in $\Omega$ is given by the following forms:
\begin{equation}\label{lnsk}
\left\{
\begin{aligned}
	&\pd_t \rho + \rho_*\dv \bu = d & &\quad\text{in $\Omega_0$, $t>0$}, \\
	&\rho_* \pd_t \bu - \DV (\bS(\bu)+ \kappa \rho_* \Delta \rho) + P'(\rho_*)\nabla \rho =\bff& &\quad\text{in $\Omega_0$, $t>0$}, \\
	&(\bS(\bu) + \kappa\rho_* \Delta \rho - P'(\rho_*) \rho \bI)\bn  - \sigma (\Delta_{\Gamma_0} + \bB) h \bn=\bg
	& &\quad\text{on $\Gamma_0$, $t>0$}, \\
	&\bn \cdot \nabla \rho=k & &\quad\text{on $\Gamma_0$, $t>0$}, \\	
	&\pd_t h -\bu \cdot \bn=\zeta& &\quad\text{on $\Gamma_0$, $t>0$}, \\	
	&(\rho, \bu)|_{t=0} = (\rho_0, \bu_0)& &\quad\text{in $\Omega_0$},
\end{aligned}
\right.
\end{equation}
where $\Delta_{\Gamma_0}$ is the Laplace-Beltrami operator on $\Gamma_0$; the right member $(d, \bff, \bg, k, \zeta)$, initial data $(\rho_0, \bu_0)$, and $\bB$ are given functions.
Since $\bB$ can be treated by the perturbation method, we consider the linearized problem \eqref{lnsk} excluding $\bB$, then
we have the resolvent problem \eqref{r1} by the Laplace transform if $\Omega_0=\R^N_+$ and $\Gamma_0=\R^N_0$.

Concerning a free boundary value problem, Saito \cite{S} considered the resolvent problem arising from a free boundary value problem without surface tension in $\R^N_+$;
namely, $\sigma=0$ in \eqref{nsk}.
He constructed $\CR$-bounded operator families satisfying the resolvent problems in $\R^N$ and $\R^N_+$. 

In this paper, we discuss the existence of the $\CR$-bounded operator families for the resolvent problem \eqref{r1}. 
Once we obtain $\CR$-boundedness for the solution operator families, we can consider the maximal $L_p$-$L_q$ regularity for the linearized problem
by the Weis's
operator valued Fourier multiplier theorem \cite{W},
which is the key estimate when we consider the local solvability for the nonlinear problem in the maximal $L_p$-$L_q$ regularity class.
Here we introduce
the definition of $\CR$-boundedness of operator families.

\begin{definition}\label{dfn2}
Let $X$ and $Y$ be Banach spaces, and let $\CL(X,Y)$ be the set of 
all bounded linear operators from $X$ into $Y$.
A family of operators $\CT \subset \CL(X,Y)$ is called $\CR$-bounded 
on $\CL(X,Y)$, if there exist constants $C > 0$ and $p \in [1,\infty)$ 
such that for any $n \in \BN$, $\{T_{j}\}_{j=1}^{n} \subset \CT$,
$\{f_{j}\}_{j=1}^{n} \subset X$ and sequences $\{r_{j}\}_{j=1}^{n}$
 of independent, symmetric, $\{-1,1\}$-valued random variables on $[0,1]$, 
we have  the inequality:
$$
\bigg \{ \int_{0}^{1} \|\sum_{j=1}^{n} r_{j}(u)T_{j}f_{j}\|_{Y}^{p}\,du
 \bigg \}^{1/p} \leq C\bigg\{\int^1_0
\|\sum_{j=1}^n r_j(u)f_j\|_X^p\,du\biggr\}^{1/p}.
$$ 
The smallest such $C$ is called $\CR$-bound of $\CT$, 
which is denoted by $\CR_{\CL(X,Y)}(\CT)$.
\end{definition}
Concerning $\CR$-boundedness, we introduce the following lemma proved by \cite[Proposition 3.4]{DHP}.

\begin{lemma}\label{lem:5.3}
$\thetag1$ 
Let $X$ and $Y$ be Banach spaces, 
and let $\CT$ and $\CS$ be $\CR$-bounded families in $\CL(X, Y)$. 
Then $\CT+\CS=\{T+S \mid T\in \CT, S\in \CS\}$ is also 
$\CR$-bounded family in $\CL(X, Y)$ and 
\[
\CR_{\CL(X, Y)}(\CT+\CS)\leq \CR_{\CL(X, Y)}(\CT)
+\CR_{\CL(X, Y)}(\CS).
\]
$\thetag2$
Let $X$, $Y$ and $Z$ be Banach spaces and 
let $\CT$ and $\CS$ be $\CR$-bounded families
 in $\CL(X, Y)$ and $\CL(Y, Z)$, respectively. 
Then $\CS\CT=\{ST \mid T\in \CT, S\in \CS\}$ is also 
an $\CR$-bounded family 
in $\CL(X, Z)$ and 
\[
\CR_{\CL(X, Z)}(\CS\CT)\leq \CR_{\CL(X, Y)}(\CT)\CR_{\CL(Y, Z)}(\CS). 
\]
\end{lemma}

\subsection{Notation}
We summarize several symbols and functional spaces used 
throughout the paper.
Let $\BN$, $\BR$ and $\BC$ denote the sets of 
all natural numbers, real numbers, and complex numbers, respectively. 
We use boldface letters, e.g. $\bu$ to 
denote vector-valued functions. 

For scalar function $f$ and $N$-vector functions $\bg$, we set
\begin{align*}
\nabla f &= (\pd_1f,\ldots,\pd_Nf)^\mathsf{T},&
\enskip 
\nabla^2 f &= (\pd_i\pd_j f)_{1\le i, j\le N},\\
\nabla^3 f &= \{\pd_i \pd_j \pd_k f \mid i, j, k = 1,\ldots, N \},\\ 
\nabla \bg &= (\pd_i g_j)_{1\le i, j\le N},&
\enskip
\nabla^2 \bg &= \{\pd_i \pd_j g_k \mid i, j, k = 1,\ldots, N \},&
\end{align*} 
where $\pd_i = \pd/\pd x_i$.

Let $\N_0=\N \cup \{0\}$.
For multi-index $\alpha'=(\alpha_1, \ldots, \alpha_{N-1}) \in \N_0^{N-1}$ and scalar function $f=f(\xi_1, \ldots, \xi_{N-1})$,
\[
	\pd^{\alpha'}_{\xi'} f = \frac{\pd^{|\alpha'|}}{\pd \xi_1^{\alpha_1} \cdots \pd \xi_{N-1}^{\alpha_{N-1}}}f,
\quad |\alpha'|=\alpha_1 + \cdots +\alpha_{N-1}.
\]

For complex valued functions $f=f(x)$ and $g=g(x)$;
$N$-vector functions $\bff=(f_1(x), \ldots, f_N(x))$ and $\bg=(g_1(x), \ldots, g_N(x))$, 
the inner products $(f, g)_{\R^N_+}$, $(f, g)_{\R^N_0}$, $(\bff, \bg)_{\R^N_+}$, and 
$(\bff, \bg)_{\R^N_0}$ are defined by
\begin{align*}
(f, g)_{\R^N_+}&=\int_{\R^N_+} f(x) \overline{g(x)}\,dx, &
(f, g)_{\R^N_0}&=\int_{\R^N_0} f(x) \overline{g(x)}\,dx',\\
(\bff, \bg)_{\R^N_+}&=\sum^N_{j=1} (f_j, g_j)_{\R^N_+}, &
(\bff, \bg)_{\R^N_0}&=\sum^N_{j=1} (f_j, g_j)_{\R^N_0},
\end{align*}
where $\overline{g(x)}$ is the complex conjugate of $g(x)$.

For Banach spaces $X$ and $Y$, $\CL(X,Y)$ denotes the set of 
all bounded linear operators from $X$ into $Y$,
$\CL(X)$ is the abbreviation of $\CL(X, X)$, and 
$\rm{Hol}\,(U, \CL(X,Y))$ denotes
 the set of all $\CL(X,Y)$ valued holomorphic 
functions defined on a domain $U$ in $\BC$. 

For any $1 < q < \infty$, $m \in \N$, 
$L_q(\BR^N_+)$ and $H_q^m(\BR^N_+)$ 
denote the usual Lebesgue space and Sobolev space;
while $\|\cdot\|_{L_q(\R^N_+)}$, $\|\cdot\|_{H_q^m(\R^N_+)}$
denote their norms, respectively; 
$W^{m+s}_q(\R^N_0) = (H^m_q(\R^N_0), H^{m+1}_q(\R^N_0))_{s, q}$ for $m \in \N_0$ and $0<s<1$, where $(\cdot, \cdot)_{s, q}$ denotes the real interpolation functor;
$C^\infty((a, b))$ denotes the set of all $C^\infty$ functions defined on $(a, b)$. 
The $d$-product space of $X$ is defined by 
$X^d=\{f=(f, \ldots, f_d) \mid f_i \in X \, (i=1,\ldots,d)\}$,
while its norm is denoted by 
$\|\cdot\|_X$ instead of $\|\cdot\|_{X^d}$ for the sake of 
simplicity. 

For the Banach space X, we also denote the usual Lebesgue space and Sobolev space of $X$-valued functions 
defined on time interval $I$ by $L_p(I, X)$ and $H^m_p(I, X)$ with $m \in \N$; 
while $\|\cdot\|_{L_p(I, X)}$, $\|\cdot\|_{H_p^m(I, X)}$
denote their norms, respectively.

Let $\gamma \ge 1$.
Set
\begin{align*}
L_{p, \gamma}(\R, X) & = \{f(t) \in L_{p, {\rm loc}}
(\R, X) \mid e^{-\gamma t}f(t) \in L_p(\R, X)\}, \\
L_{p, \gamma, 0}(\R, X) & = \{f(t) \in L_{p, \gamma}(\R, X) 
\mid
f(t) = 0 \enskip (t < 0)\}, \\
H^m_{p, \gamma}(\R, X) & = \{f(t) \in
 L_{p, \gamma}(\R, X) \mid 
e^{-\gamma t}\pd_t^j f(t) \in L_p(I, X)
\enskip (j=1, \ldots, m)\}, \\
H^m_{p, \gamma, 0}(\R, X) & = H^m_{p, \gamma}(\R, X)
\cap L_{p, \gamma, 0}(\R, X).
\end{align*}
Let $\CL$ and $\CL^{-1}$ denote the Laplace transform and the
Laplace inverse transform, respectively, which are defined by 
\[
\CL[f](\lambda) = \int_{\R} e^{-\lambda t} f(t) \,dt, \quad
\CL^{-1}[g](t) = \frac{1}{2\pi}
\int_{\R} e^{\lambda t} g(\tau) \,d\tau
\]
with $\lambda = \gamma + i\tau \in \C$.  Given $s \ge 0$
and $X$-valued function $f(t)$, we set 
\[
(\Lambda^s_\gamma f)(t) = \CL^{-1} [\lambda^s
\CL[f](\lambda)](t).
\]
The Bessel potential space of $X$-valued functions 
of order $s$ is defined by the following:
\begin{align*}
H^s_{p, \gamma}(\R, X) & = \{f \in L_p(\R, X) \mid 
e^{-\gamma t} (\Lambda_\gamma^s f)(t) \in L_p(\R, X)\}, \\
H^s_{p, \gamma, 0}(\R, X) & = \{f \in H^s_{p, \gamma}
(\R, X) \mid f(t) = 0 \enskip(t < 0)\}.
\end{align*}

The letter $C$ denotes generic constants and the constant 
$C_{a,b,\ldots}$ depends on $a,b,\ldots$. 
The values of constants $C$ and $C_{a,b,\ldots}$ 
may change from line to line.

This paper is organized as follows.
In the next section, we state the main theorem concerning the $\CR$-bounded operator families arising from the free boundary value problem for the compressible fluid model of Korteweg type with surface tension. 
Section \ref{sec:idea} and Section \ref{sec:rbdd}
prove the main theorem. 
As preliminaries, we study the reduced resolvent problem by using the result for the Navier-Stokes-Korteweg system without surface tension.
Section \ref{general} proves the existence of the $\CR$-bounded operator families
for the system containing lower-order derivatives.
As an application of the main theorem, Section \ref{mr} proves the maximal $L_p$-$L_q$ regularity for the linearized equations.

\section{Main theorem}\label{sec:main}
Let us consider the following rescaled problem:
\begin{equation}\label{r2}
\left\{
\begin{aligned}
	&\lambda \rho + \dv \bu = d & \quad&\text{in $\R^N_+$}, \\
	&\lambda \bu - \DV\{\bS(\bu) - (\gamma_* - \kappa \Delta) \rho\bI\}=\bff& \quad&\text{in $\R^N_+$},\\
	&\{\bS(\bu) -(\gamma_* - \kappa \Delta) \rho \bI\} \bn
	 - \sigma \Delta' h \bn
	=\bg & \quad&\text{on $\R^N_0$},\\
	&\bn \cdot \nabla \rho = k & \quad&\text{on $\R^N_0$},\\
	&\lambda h - \bu \cdot \bn = \zeta & \quad&\text{on $\R^N_0$},
\end{aligned}
\right.
\end{equation}
where we have
replaced in \eqref{r1} $(\rho, \mu, \nu, \kappa, \sigma)$ with 
$(\rho_* \rho, \rho_* \mu, \rho_* \nu, \kappa/\rho_*, \rho_* \sigma)$.
Hereafter, we mainly consider the system \eqref{r2} in the case $\gamma_* = 0$; namely,
\begin{equation}\label{main prob}
\left\{
\begin{aligned}
	&\lambda \rho + \dv \bu = d & \quad&\text{in $\R^N_+$}, \\
	&\lambda \bu - \DV(\bS(\bu) + \kappa \Delta \rho\bI)=\bff& \quad&\text{in $\R^N_+$},\\
	&(\bS(\bu) + \kappa \Delta \rho \bI) \bn
	 - \sigma \Delta' h \bn
	=\bg & \quad&\text{on $\R^N_0$},\\
	&\bn \cdot \nabla \rho = k & \quad&\text{on $\R^N_0$},\\
	&\lambda h - \bu \cdot \bn = \zeta & \quad&\text{on $\R^N_0$}.
\end{aligned}
\right.
\end{equation}
Note that the result for \eqref{r2} in the case $\gamma_* \in \R$ can be obtained by 
the result in the case $\gamma_*=0$.
We discuss it in more detail in Sec. \ref{general} below.

Let $\bF=(d, \bff, \bg, k, \zeta)$.
To state the main theorem, we define a function space for the right member $\bF$ as 
\[
	X_q(\R^N_+) = H^1_q(\R^N_+) \times L_q(\R^N_+)^N \times H^1_q(\R^N_+) ^N
	\times H^2_q(\R^N_+) \times W^{2-1/q}_q(\R^N_0).
\]
For $\bF \in X_q(\R^n_+)$, we set
\begin{align*}
	\CX_q(\R^N_+) &=H^1_q(\R^N_+) \times L_q(\R^N_+)^\CN \times W^{2-1/q}_q(\R^N_0),
	\quad \CN = N+N^2+N+N^2+N+1,\\
	\CF_\lambda \bF &=(d, \bff, \nabla \bg, \lambda^{1/2} \bg, \nabla^2 k, \nabla \lambda^{1/2}k, \lambda k, \zeta) \in \CX_q(\R^N_+).
\end{align*}
Moreover, we define symbols for the solution $(\rho, \bu, h)$ of \eqref{main prob} as
\begin{align*}
	\fA_q(\R^N_+) &= L_q(\R^N_+)^{N^3+N^2} \times H^1_q(\R^N_+),
	&\CR_\lambda \rho &=(\nabla^3 \rho, \lambda^{1/2}\nabla^2\rho, \lambda \rho),\\
	\fB_q(\R^N_+) &= L_q(\R^N_+)^{N^3+N^2+N},
	&\CS_\lambda \bu &=(\nabla^2 \bu, \lambda^{1/2}\nabla \bu, \lambda \bu),\\
	\fC_q(\R^N_0) &= W^{3-1/q}_q(\R^N_0) \times W^{2-1/q}_q(\R^N_0),
	&\CT_\lambda h &=(h, \lambda h).
\end{align*}
To describe the sectorial region, we also set an angle due to analysis for the whole space problem (cf. \cite[Sec. 2]{S}). 
Let $\alpha$ be a constant given by
\[
	\alpha = \left(\frac{\mu + \nu}{2\kappa}\right)^2 - \frac{1}{\kappa},
\]
and let $\tilde \epsilon_* \in [0, \pi/2)$ be an angle defined as
\begin{equation}\label{angle}
	\tilde \epsilon_* =
\left\{
\begin{aligned}
&0&(\alpha \ge 0),\\
&\arg \left(\frac{\mu+\nu}{2\kappa} + i\sqrt{|\alpha|}\right) &(\alpha <0).
\end{aligned}
\right. 
\end{equation}
Our main result for the system \eqref{main prob} is as follows.

\begin{theorem}\label{main}
Let $1<q<\infty$.
Assume that $\mu$, $\nu$, $\kappa$, and $\sigma$
are positive constants satisfying 
\begin{equation}\label{condi}
\alpha \neq 0, \quad \kappa \neq \mu \nu.
\end{equation}
Then there is a constant $\epsilon_* \in (\tilde \epsilon_*, \pi/2)$ such that for any $\epsilon \in (\epsilon_*, \pi/2)$
there exists a constant $\lambda_0 \ge 1$ such that the following assertions hold true: 

$\thetag1$ 
For any $\lambda \in \Sigma_{\epsilon, \lambda_0}$ there exist operator families 
\begin{align*}
&\CA_0 (\lambda) \in 
{\rm Hol} (\Sigma_{\epsilon, \lambda_0}, 
\CL(\CX_q(\R^N_+), H^3_q(\R^N_+)))\\
&\CB_0 (\lambda) \in 
{\rm Hol} (\Sigma_{\epsilon, \lambda_0}, 
\CL(\CX_q(\R^N_+), H^2_q(\R^N_+)^N)),\\
&\CC_0 (\lambda) \in 
{\rm Hol} (\Sigma_{\epsilon, \lambda_0}, 
\CL(\CX_q(\R^N_+), W^{3-1/q}_q(\R^N_0)))
\end{align*}
such that 
for any $\bF=(d, \bff, \bg, k, \zeta) \in X_q(\R^N_+)$, 
\begin{equation*}
\rho = \CA_0 (\lambda) \CF_\lambda \bF, \quad
\bu = \CB_0 (\lambda) \CF_\lambda \bF, \quad
h = \CC_0 (\lambda) \CF_\lambda \bF
\end{equation*}
are unique solutions of the problem \eqref{main prob}.

$\thetag2$ 
There exists a positive constant $r$ such that
\begin{equation} \label{rbdd1}
\begin{aligned}
&\CR_{\CL(\CX_q(\R^N_+), \fA_q(\R^N_+))}
(\{(\tau \pd_\tau)^n \CR_\lambda \CA_0 (\lambda) \mid 
\lambda \in \Sigma_{\epsilon, \lambda_0}\}) 
\leq r,\\
&\CR_{\CL(\CX_q(\R^N_+), \fB_q(\R^N_+))}
(\{(\tau \pd_\tau)^n \CS_\lambda \CB_0 (\lambda) \mid 
\lambda \in \Sigma_{\epsilon, \lambda_0}\}) 
\leq r,\\
&\CR_{\CL(\CX_q(\R^N_+), \fC_q(\R^N_0))}
(\{(\tau \pd_\tau)^n \CT_\lambda \CC_0 (\lambda) \mid 
\lambda \in \Sigma_{\epsilon, \lambda_0}\}) 
\leq r
\end{aligned}
\end{equation}
for $n = 0, 1$.
Here above constants $\lambda_0$ and $r$ depend solely on $N$, $q$, $\epsilon$,
$\mu$, $\nu$, $\kappa$, and $\sigma$. 
\end{theorem}

Theorem \ref{main} can be proved by the following theorem, which will be discussed in Sec. \ref{sec:idea} and Sec. \ref{sec:rbdd} below.

\begin{theorem}\label{main'}
Let $1<q<\infty$.
Assume that $\mu$, $\nu$, $\kappa$, and $\sigma$ are positive constants satisfying \eqref{condi}.
Let $\epsilon \in (\epsilon_*, \pi/2)$ for $\epsilon_*$ given in Theorem \ref{main}. 
Set 
\begin{align*}
	Y_q(\R^N_+) &=  H^1_q(\R^N_+) \times L_q(\R^N_+)^N \times H^1_q(\R^N_+) ^N
	\times H^2_q(\R^N_+) \times H^2_q(\R^N_+),\\
	\CY_q(\R^N_+) &=H^1_q(\R^N_+) \times L_q(\R^N_+)^\CN \times H^2_q(\R^N_+),
	\quad \CN = N+N^2+N+N^2+N+1,\\
	\tilde \fC_q(\R^N_+) &= H^3_q(\R^N_+) \times H^2_q(\R^N_+).
\end{align*}
Then there exists a constant $\lambda_0 \ge 1$ such that the following assertions hold true:

$\thetag1$ 
For any $\lambda \in \Sigma_{\epsilon, \lambda_0}$ there exist operator families 
\begin{align*}
&\CA (\lambda) \in 
{\rm Hol} (\Sigma_{\epsilon, \lambda_0}, 
\CL(\CY_q(\R^N_+), H^3_q(\R^N_+)))\\
&\CB (\lambda) \in 
{\rm Hol} (\Sigma_{\epsilon, \lambda_0}, 
\CL(\CY_q(\R^N_+), H^2_q(\R^N_+)^N)),\\
&\CC (\lambda) \in 
{\rm Hol} (\Sigma_{\epsilon, \lambda_0}, 
\CL(\CY_q(\R^N_+), H^3_q(\R^N_+)))
\end{align*}
such that 
for any $\bF=(d, \bff, \bg, k, \zeta) \in Y_q(\R^N_+)$, 
\begin{equation*}
\rho = \CA (\lambda) \CF_\lambda \bF, \quad
\bu = \CB (\lambda) \CF_\lambda \bF, \quad
h = \CC (\lambda) \CF_\lambda \bF
\end{equation*}
are solutions of problem \eqref{main prob}.

$\thetag2$ 
There exists a positive constant $r$ such that
\[
\begin{aligned}
&\CR_{\CL(\CY_q(\R^N_+), \fA_q(\R^N_+))}
(\{(\tau \pd_\tau)^n \CR_\lambda \CA (\lambda) \mid 
\lambda \in \Sigma_{\epsilon, \lambda_0}\}) 
\leq r,\\
&\CR_{\CL(\CY_q(\R^N_+), \fB_q(\R^N_+))}
(\{(\tau \pd_\tau)^n \CS_\lambda \CB (\lambda) \mid 
\lambda \in \Sigma_{\epsilon, \lambda_0}\}) 
\leq r,\\
&\CR_{\CL(\CY_q(\R^N_+), \tilde \fC_q(\R^N_+))}
(\{(\tau \pd_\tau)^n \CT_\lambda \CC (\lambda) \mid 
\lambda \in \Sigma_{\epsilon, \lambda_0}\}) 
\leq r
\end{aligned}
\]
for $n = 0, 1$.
Here, above constants $\lambda_0$ and $r$ depend solely on $N$, $q$, $\epsilon$,  
$\mu$, $\nu$, $\kappa$, and $\sigma$. 
\end{theorem}

Admitting Theorem \ref{main'}, we prove Theorem \ref{main}. 

\begin{proof}[proof of Theorem \ref{main}]
Let $\bF\in X_q(\R^N_+)$.
Note that there exist linear mappings
\[
	\CT : H^n_q(\R^N_+) \to W^{n-1/q}_q(\R^N_0),
\quad
	\CE : W^{n-1/q}_q(\R^N_0) \to H^n_q(\R^N_+)
\]
such that $\|\CT f\|_{W^{n-1/q}_q(\R^N_0)} \le C\|f\|_{W^n_q(\R^N_+)}$
and $\|\CE g\|_{W^n_q(\R^N_+)}\le C\|g\|_{W^{n-1/q}_q(\R^N_0)}$ for $n=2, 3$ (cf. \cite[Proposition 9.5.4]{S2016}).
For $\bF \in X_q (\R^N_+)$, we define operator $\tilde \CE$ as
$\tilde \CE \CF_\lambda \bF=(d, \bff, \nabla \bg, \lambda^{1/2} \bg, \nabla^2 k, \nabla \lambda^{1/2}k, \lambda k, \CE \zeta) \in \CY_q (\R^N_+)$. 
Setting $\CA_0(\lambda) = \CA(\lambda) \tilde \CE$, 
$\CB_0(\lambda) = \CB(\lambda) \tilde \CE$, and 
$\CC_0(\lambda) = \CT \CC(\lambda) \tilde \CE$,
Theorem \ref{main'} implies the existence of solutions to \eqref{main prob}
and $\CA_0(\lambda)$, $\CB_0(\lambda)$, and $\CC_0(\lambda)$ satisfy \eqref{rbdd1}.

The uniqueness of the solution to \eqref{main prob} follows from the existence theorem of the dual problem.
In fact, we assume that $(\rho, \bu, h)$ satisfies the following system, which is \eqref{main prob} with the right member $(d, \bff, \bg, h, \zeta)$ vanishing.
\begin{equation}\label{homo}
\left\{
\begin{aligned}
	&\lambda \rho + \dv \bu = 0 & \quad&\text{in $\R^N_+$}, \\
	&\lambda \bu - \DV(\bS(\bu) + \kappa \Delta \rho\bI)=0& \quad&\text{in $\R^N_+$},\\
	&(\bS(\bu) + \kappa \Delta \rho \bI) \bn
	 - \sigma \Delta' h \bn
	=0 & \quad&\text{on $\R^N_0$},\\
	&\bn \cdot \nabla \rho = 0 & \quad&\text{on $\R^N_0$},\\
	&\lambda h - \bu \cdot \bn = 0 & \quad&\text{on $\R^N_0$}.
\end{aligned}
\right.
\end{equation}
For any $\lambda \in \Sigma_{\epsilon, \lambda_0}$, $\bP \in C^\infty_0(\R^N_+)^N$, 
there exist a solution $(\theta, \bv, k) \in H^3_{q'}(\R^N_+) \times H^2_{q'}(\R^N_+)^N \times W^{3-1/q'}_{q'}(\R^N_+)$ such that
\begin{equation*}\label{dual}
\left\{
\begin{aligned}
	&\bar \lambda \theta + \dv \bv = 0 & \quad&\text{in $\R^N_+$}, \\
	&\bar \lambda \bv - \DV(\bS(\bv) + \kappa \Delta \theta \bI)=\bP& \quad&\text{in $\R^N_+$},\\
	&(\bS(\bv) + \kappa \Delta \theta \bI) \bn
	 - \sigma \Delta' k  \bn
	=0 & \quad&\text{on $\R^N_0$},\\
	&\bn \cdot \nabla \theta = 0 & \quad&\text{on $\R^N_0$},\\
	&\bar \lambda k - \bv \cdot \bn = 0 & \quad&\text{on $\R^N_0$}.
\end{aligned}
\right.
\end{equation*}
Note the facts that
\[
	 \sum^N_{i, j=1} \int_{\R^N_+} \pd_j u_i \overline{(\bS(\bv)_{ij} + \Delta \theta \delta_{ij})} \,dx 
	= 	\sum^N_{i, j=1} \int_{\R^N_+} (\bS(\bu)_{ij} + \Delta \rho \delta_{ij}) \overline{\pd_j v_i}\,dx,
\]
where we have used $\bS(\bu)_{ij}=\mu(\pd_i u_j + \pd_j u_j)+(\nu-\mu)\dv\bu \delta_{ij}$ and 
the boundary condition $\bn \cdot \nabla \rho = \bn \cdot \nabla \theta = 0$.
Then we have
\begin{align*}
	(\bu, \bP)_{\R^N_+}
	&=(\bu, \bar \lambda \bv - \DV(\bS(\bv) + \kappa \Delta \theta \bI))_{\R^N_+}\\
	&=(\lambda \bu, \bv)_{\R^N_+} - (\bu, (\bS(\bv) + \kappa \Delta \theta \bI)\bn)_{\R^N_0}
	+ \sum^N_{i, j=1} \int_{\R^N_+}\pd_j u_i \overline{(\bS(\bv)_{ij}+\kappa \Delta \theta \delta_{ij})} \,dx\\
	&=(\DV(\bS(\bu) + \kappa \Delta \rho \bI), \bv)_{\R^N_+}
	- (\bu, \sigma \Delta'k\bn)_{\R^N_0}
	+ \sum^N_{i, j=1} \int_{\R^N_+} (\bS(\bu)_{ij} + \kappa \Delta \rho \delta_{ij}) \overline{\pd_j v_i}\,dx\\
	&=((\bS(\bu) + \kappa \Delta \rho \bI)\bn, \bv)_{\R^N_0} 
	- (\bu, \sigma \Delta' k \bn)_{\R^N_0}
	=(\sigma \Delta' h, \bv \cdot \bn)_{\R^N_0} 
	- (\bu \cdot \bn, \sigma \Delta'k)_{\R^N_0}\\
	&=(\sigma \Delta' h, \bar \lambda k )_{\R^N_0} 
	- (\lambda h, \sigma \Delta'k)_{\R^N_0}
	=(\sigma \Delta' h, \bar \lambda k )_{\R^N_0} 
	- (\sigma \Delta' h, \bar \lambda k)_{\R^N_0}\\
	&=0
\end{align*}
for arbitrary $\bP$, 
which implies that $\bu=0$ in $\R^N_+$.
Then by the last equation of \eqref{homo} 
we have $h=0$ on $\R^N_0$ because $\lambda \neq 0$.
Furthermore, the first equation of \eqref{homo} implies that $\lambda \rho =0$, then we have $\rho=0$ in $\R^N_+$ by $\lambda \neq 0$.
This completes the proof of Theorem \ref{main}.
\end{proof}

\section{The main idea and preliminaries}\label{sec:idea}
\subsection{The main idea}
Assume that $(\theta, \bv)$ and $(\omega, \bw, h)$ satisfy the following systems, respectively.
\begin{equation}\label{r5}
\left\{
\begin{aligned}
	&\lambda \theta + \dv \bv = d & \quad&\text{in $\R^N_+$}, \\
	&\lambda \bv -\DV(\bS(\bv)+ \kappa \Delta \theta\bI)=\bff& \quad&\text{in $\R^N_+$},\\
	&(\bS(\bv)+ \kappa \Delta \theta\bI) \bn=\bg &\quad&\text{on $\R^N_0$},\\
	&\bn \cdot \nabla \theta = k & \quad&\text{on $\R^N_0$},
\end{aligned}
\right.
\end{equation}
\begin{equation}\label{r6}
\left\{
\begin{aligned}
	&\lambda \omega + \dv \bw = 0 & \quad&\text{in $\R^N_+$}, \\
	&\lambda \bw - \DV(\bS(\bw)+ \kappa \Delta \omega\bI)=0& \quad&\text{in $\R^N_+$},\\
	&(\bS(\bw)+ \kappa \Delta \omega\bI) \bn 
	- \sigma \Delta' h \bn=0 & \quad&\text{on $\R^N_0$},\\
	&\bn \cdot \nabla \omega = 0 & \quad&\text{on $\R^N_0$},\\
	&\lambda h - \bw \cdot \bn = \zeta + \bv \cdot \bn & \quad&\text{on $\R^N_0$}.
\end{aligned}
\right.
\end{equation}
Setting
$\rho= \theta + \omega$ and $\bu = \bv + \bw$, 
we see that $(\rho, \bu, h)$ satisfies \eqref{main prob}.
The system \eqref{r5} was studied by \cite[Theorem 1.3]{S} as follows.
\begin{proposition}\label{without}
Let $1<q<\infty$ and $\lambda_0>0$.
Assume that $\mu$, $\nu$, and $\kappa$
are positive constants satisfying \eqref{condi};
$\tilde \epsilon_*$ is a constant given by \eqref{angle}.
Set 
\begin{align*}
	Z_q(\R^N_+) &= H^1_q(\R^N_+) \times L_q(\R^N_+)^N \times H^1_q(\R^N_+) ^N
	\times H^2_q(\R^N_+),\\
	\bG &=(d, \bff, \bg, k)\in Z_q(\R^N_+),\\
	\CZ_q(\R^N_+) &=H^1_q(\R^N_+) \times L_q(\R^N_+)^\CN,
	\quad \CN = N+N^2+N+N^2+N+1,\\
	\CG_\lambda \bG &=(d, \bff, \nabla \bg, \lambda^{1/2} \bg, \nabla^2 k, \nabla \lambda^{1/2}k, \lambda k) \in \CZ_q(\R^N_+).
\end{align*}
Then there is a constant $\epsilon_* \in (\tilde \epsilon_*, \pi/2)$ such that for any $\epsilon \in (\epsilon_*, \pi/2)$
the following assertions hold true: 

$\thetag1$ 
For any $\lambda \in \Sigma_{\epsilon, \lambda_0}$ there exist operator families 
\begin{align*}
&\CA_1 (\lambda) \in 
{\rm Hol} (\Sigma_{\epsilon, \lambda_0}, 
\CL(\CZ_q(\R^N_+), H^3_q(\R^N_+)))\\
&\CB_1 (\lambda) \in 
{\rm Hol} (\Sigma_{\epsilon, \lambda_0}, 
\CL(\CZ_q(\R^N_+), H^2_q(\R^N_+)^N))
\end{align*}
such that 
for any $\bF=(d, \bff, \bg, k) \in Z_q(\R^N_+)$, 
\begin{equation*}
\rho = \CA_1 (\lambda) \CG_\lambda \bG, \quad
\bu = \CB_1 (\lambda) \CG_\lambda \bG
\end{equation*}
are unique solutions of problem \eqref{r5}.

$\thetag2$ 
There exists a positive constant $r$ such that
\begin{equation*}
\begin{aligned}
&\CR_{\CL(\CZ_q(\R^N_+), \fA_q(\R^N_+))}
(\{(\tau \pd_\tau)^n \CR_\lambda \CA_1 (\lambda) \mid 
\lambda \in \Sigma_{\epsilon, \lambda_0}\}) 
\leq r,\\
&\CR_{\CL(\CZ_q(\R^N_+), \fB_q(\R^N_+))}
(\{(\tau \pd_\tau)^n \CS_\lambda \CB_1 (\lambda) \mid 
\lambda \in \Sigma_{\epsilon, \lambda_0}\}) 
\leq r
\end{aligned}
\end{equation*}
for $n = 0, 1$.
Here, above constant $r$ depend solely on $N$, $q$, $\epsilon$, $\lambda_0$, $\mu$, $\nu$, $\kappa$, and $\sigma$. 
\end{proposition}

\begin{remark}
The constant $\epsilon_*$ will be determined in Lemma \ref{est:fl} $\thetag3$ below.
\end{remark}
In view of Prop \ref{without}, it suffices to consider the following system arising \eqref{r6}.
\begin{equation}\label{r7}
\left\{
\begin{aligned}
	&\lambda \rho + \dv \bu = 0 & \quad&\text{in $\R^N_+$}, \\
	&\lambda \bu - \mu \Delta \bu - \nu \nabla \dv \bu 
	- \kappa \Delta \nabla \rho=0& \quad&\text{in $\R^N_+$},\\
	&\{\mu \bD(\bu) + (\nu-\mu)\dv \bu \bI + \kappa \Delta \rho \bI\} \bn 
	- \sigma \Delta' h \bn=0 & \quad&\text{on $\R^N_0$},\\
	&\bn \cdot \nabla \rho = 0 & \quad&\text{on $\R^N_0$},\\
	&\lambda h - \bu \cdot \bn = \eta & \quad&\text{on $\R^N_0$},
\end{aligned}\right.
\end{equation}
where we have used $\DV \bS(\bu)=\mu \Delta \bu + \nu \nabla \dv \bu$.
For \eqref{r7}, we will prove the following result from the next subsection.
\begin{proposition}\label{with}
Let $1<q<\infty$.
Assume that $\mu$, $\nu$, $\kappa$, and $\sigma$
are positive constants satisfying \eqref{condi}.
Let $\epsilon \in (\epsilon_*, \pi/2)$ for $\epsilon_*$ given in Proposition \ref{without}. 
Then 
there exists a constant $\lambda_0 \ge 1$ such that the following assertions hold true:

$\thetag1$ 
For any $\lambda \in \Sigma_{\epsilon, \lambda_0}$ there exist operator families 
\begin{equation}\label{family}
\begin{aligned}
&\CA_2 (\lambda) \in 
{\rm Hol} (\Sigma_{\epsilon, \lambda_0}, 
\CL(H^2_q(\R^N_+), H^3_q(\R^N_+)))\\
&\CB_2 (\lambda) \in 
{\rm Hol} (\Sigma_{\epsilon, \lambda_0}, 
\CL(H^2_q(\R^N_+), H^2_q(\R^N_+)^N)),\\
&\CC_2 (\lambda) \in 
{\rm Hol} (\Sigma_{\epsilon, \lambda_0}, 
\CL(H^2_q(\R^N_+), H^3_q(\R^N_+)))
\end{aligned}
\end{equation}
such that 
for any $\eta \in H^2_q(\R^N_+)$, 
\begin{equation*}
\rho = \CA_2 (\lambda) \eta, \quad
\bu = \CB_2 (\lambda) \eta, \quad
h = \CC_2 (\lambda) \eta
\end{equation*}
are solutions of problem \eqref{r7}.

$\thetag2$ 
There exists a positive constant $r$ such that
\begin{equation} \label{est:rbdd}
\begin{aligned}
&\CR_{\CL(H^2_q(\R^N_+), \fA_q(\R^N_+))}
(\{(\tau \pd_\tau)^n \CR_\lambda \CA_2 (\lambda) \mid 
\lambda \in \Sigma_{\epsilon, \lambda_0}\}) 
\leq r,\\
&\CR_{\CL(H^2_q(\R^N_+), \fB_q(\R^N_+))}
(\{(\tau \pd_\tau)^n \CS_\lambda \CB_2 (\lambda) \mid 
\lambda \in \Sigma_{\epsilon, \lambda_0}\}) 
\leq r,\\
&\CR_{\CL(H^2_q(\R^N_+), \tilde \fC_q(\R^N_+))}
(\{(\tau \pd_\tau)^n \CT_\lambda \CC_2 (\lambda) \mid 
\lambda \in \Sigma_{\epsilon, \lambda_0}\}) 
\leq r
\end{aligned}
\end{equation}
for $n = 0, 1$.
Here, above constants $\lambda_0$ and $r$ depend solely on $N$, $q$, $\epsilon$, 
$\mu$, $\nu$, $\kappa$, and $\sigma$. 
\end{proposition}
Let us point out that Theorem \ref{main'} is obtained by Proposition \ref{without} and Proposition \ref{with}.

\begin{proof}[Proof of Theorem \ref{main'}]
Set $\bv \cdot \bn=-v_N=-\CB_{1N}(\lambda)\CG_\lambda \bG$.
In view of \eqref{r6} and Proposition \ref{with}, we also set
\begin{align*}
	\CA_3(\lambda) \CF_\lambda \bF &=\CA_2(\lambda)(\zeta - \CB_{1N}(\lambda)\CG_\lambda \bG),\\
	\CB_3(\lambda) \CF_\lambda \bF &=\CB_2(\lambda)(\zeta - \CB_{1N}(\lambda)\CG_\lambda \bG),\\
	\CC(\lambda) \CF_\lambda \bF &=\CC_2(\lambda)(\zeta - \CB_{1N}(\lambda)\CG_\lambda \bG),
\end{align*}
and then $\omega=\CA_3(\lambda) \CF_\lambda \bF$, $\bw=\CB_3(\lambda) \CF_\lambda \bF$, and $h=\CC(\lambda) \CF_\lambda \bF$
satisfy \eqref{r6}.
Define operator families $\CA(\lambda)$ and $\CB(\lambda)$ as 
$\CA(\lambda) \CF_\lambda \bF = \CA_1(\lambda) \CF_\lambda \bF + \CA_3(\lambda) \CF_\lambda \bF$
and
$\CB(\lambda) \CF_\lambda \bF = \CB_1(\lambda) \CF_\lambda \bF + \CB_3(\lambda) \CF_\lambda \bF$.
Then we see that $\CA(\lambda)$, $\CB(\lambda)$, and $\CC(\lambda)$ are the desired operator families by
Proposition \ref{without}, Proposition \ref{with}, and Lemma \ref{lem:5.3}.
\end{proof}

Therefore, our main task is to show Proposition \ref{with}.

\subsection{Solution formulas}
The system \eqref{r7} is equivalent to 
\begin{equation}\label{r8}
\left\{
\begin{aligned}
	&\lambda \rho + \dv \bu = 0 & \quad&\text{in $\R^N_+$}, \\
	&\lambda \bu - \mu \Delta \bu - \nu \nabla \dv \bu 
	- \kappa \Delta \nabla \rho=0& \quad&\text{in $\R^N_+$},\\
	&\mu (\pd_j u_N + \pd_N u_j) = 0 & \quad&\text{on $\R^N_0$},\\
	&2\mu \pd_N u_N + (\nu-\mu)\dv \bu + \kappa \Delta \rho
	- \sigma \Delta' h =0 & \quad&\text{on $\R^N_0$},\\
	&\pd_N \rho = 0 & \quad&\text{on $\R^N_0$},\\
	&\lambda h+u_N = \eta & \quad&\text{on $\R^N_0$}
\end{aligned}
\right.
\end{equation}
for $j=1, \ldots, N-1$.
In this subsection, we calculate the solution formulas for \eqref{r8}.
For this purpose, we define the partial Fourier transform with respect to $x'=(x_1, \ldots, x_{N-1})$ and its inverse transform as
\begin{align*}
	&\hat u = \hat u(x_N) = \int_{\R^{N-1}} e^{- i x'\cdot \xi'} u(x', x_N)\,dx',\\
	&\CF^{-1}_{\xi'}[v(\xi', x_N)](x') = \frac{1}{(2\pi)^{N-1}} \int_{\R^{N-1}} e^{ i x'\cdot \xi'} v(\xi', x_N)\,d\xi'.
\end{align*}  
Set $\vp = \dv \bu$. Applying the partial Fourier transform to \eqref{r8}, we have
\begin{equation}\label{r}
\left\{
\begin{aligned}
	&\lambda \hat \rho + \hat \vp = 0, & x_N>0,\\
	&\lambda \hat u_j - \mu(\pd_N^2 - |\xi'|^2) \hat u_j -\nu i\xi_j \hat \vp - \kappa i\xi_j(\pd_N^2-|\xi'|^2)\hat \rho=0, & x_N>0,\\
	&\lambda \hat u_N - \mu(\pd_N^2 - |\xi'|^2) \hat u_N -\nu \pd_N \hat \vp - \kappa \pd_N(\pd_N^2-|\xi'|^2)\hat \rho=0, & x_N>0,\\
	&\mu(i\xi_j \hat u_N(0) + \pd_N \hat u_j(0))=0,\\
	&2\mu \pd_N \hat u_N(0) + (\nu-\mu) \hat \vp(0) + \kappa(\pd_N^2-|\xi'|^2) \hat \rho(0) + \sigma |\xi'|^2 \hat h(0)=0,\\
	&\pd_N \hat \rho(0)=0,\\
	&\lambda \hat h(0) + \hat u_N(0)=\hat \eta(0).
\end{aligned}
\right.
\end{equation}
To represent a solution, we set and recall the following notations:
\begin{equation}\label{notation}
\begin{aligned}
	&s_1=s_+, \quad s_2 = s_-, \quad
	s_\pm = \left\{
	\begin{aligned}
	&\frac{\mu+\nu}{2 \kappa} \pm \sqrt{\alpha} \enskip (\alpha > 0),\\ 
	&\frac{\mu+\nu}{2 \kappa} \pm i \sqrt{|\alpha|} \enskip (\alpha < 0),
	\end{aligned}
	\right.
	\quad \alpha = \left(\frac{\mu+\nu}{2\kappa}\right)^2 - \frac{1}{\kappa} \neq 0,\\
	&\omega_\lambda = \sqrt{|\xi'|^2 + \mu^{-1} \lambda}, \quad t_j = \sqrt{|\xi'|^2 + s_j\lambda}, \\
	&\fl_j(\xi', \lambda) = \mu^{-2} \lambda t_j (t_j + \omega_\lambda)(t_2^2 + t_2t_1 + t_1^2 - |\xi'|^2) \\
	&\qquad+ 4 \omega_\lambda |\xi'|^2 \{s_j t_j \omega_\lambda (t_j+\omega_\lambda)
	-(s_j-\mu^{-1}) t_1 t_2 (t_2+t_1)\},\\
	&\fa(\xi', \lambda) = \frac{\kappa^{-1} (t_2+t_1)}{s_2-s_1}, \\
	&\fp_j(\xi', \lambda) = (4s_j -3\mu^{-1})\omega_\lambda +\mu^{-1}t_j,
	\quad \fq_j(\xi', \lambda) = (2s_j - \mu^{-1})\omega_\lambda + \mu^{-1}t_j,\\
	&\CM_0(x_N)=\frac{e^{-t_2 x_N}-e^{-t_1 x_N}}{t_2-t_1},
	\quad \CM_j(x_N)=\frac{e^{-t_j x_N}-e^{-\omega_\lambda x_N}}{t_2-t_1}
\end{aligned}
\end{equation}
for $j=1, 2$. Here we note that $s_1s_2=\kappa^{-1}$.
Thanks to \cite[Sec.5]{S}, the solution formulas for $\hat \rho$, $\hat u_j$, and $\hat u_N$ satisfying \eqref{r} are written by $\hat h(0)$ as follows:
\begin{align}
	\hat \rho(x_N) &= \frac{t_1s_1s_2 (t_1+\omega_\lambda)(\omega_\lambda^2+|\xi'|^2)}{\mu \fl_1(\xi',\lambda)}e^{-t_1x_N}\sigma |\xi'|^2 \hat h(0)\nonumber \\
	&\enskip + \frac{s_1s_2 t_1^2(t_1+\omega_\lambda)(\omega_\lambda^2+|\xi'|^2)}{\mu \fl_1(\xi', \lambda)}\CM_0(x_N)\sigma |\xi'|^2 \hat h(0),\label{rho}\\
	\hat u_j(x_N) &=\frac{i\xi_j}{2\mu \omega_\lambda^2}e^{-\omega_\lambda x_N}\sigma |\xi'|^2 \hat h(0)\nonumber\\
	&\enskip +
	\sum^2_{l=1} \frac{(-1)^{l+1}i\xi_j t_1 t_2 (\omega_\lambda^2+|\xi'|^2)\fa(\xi', \lambda)\fp_l(\xi', \lambda)}{2\mu \omega_\lambda^2 s_l \fl_l(\xi', \lambda)}
e^{-\omega_\lambda x_N} 
	\sigma |\xi'|^2 \hat h(0)\nonumber\\
	&\enskip -  
	\sum^2_{l=1} \frac{(-1)^{l+1}i \xi_j t_1t_2 (\omega_\lambda^2 + |\xi'|^2)s_1s_2(t_l+\omega_\lambda)}{\mu s_l \fl_l(\xi, \lambda)}\CM_l(x_N)
	\sigma |\xi'|^2 \hat h(0), \label{uj}\\
	\hat u_N(x_N) &=\frac{1}{2\mu \omega_\lambda}e^{-\omega_\lambda x_N}\sigma |\xi'|^2 \hat h(0)\nonumber\\
	&\enskip + \sum^2_{l=1}\frac{(-1)^{l+1}t_1 t_2 (\omega_\lambda^2 +|\xi'|^2)\fa(\xi', \lambda) \fq_l(\xi', \lambda)}{2\mu \omega_\lambda s_l\fl_l(\xi', \lambda)}e^{-\omega_\lambda x_N}\sigma |\xi'|^2 \hat h(0) \nonumber\\
	&\enskip + 
	\sum^2_{l=1}\frac{(-1)^{l+1}t_1t_2 (\omega_\lambda^2 +|\xi'|^2)s_1s_2t_l(t_l+\omega_\lambda)}{\mu s_l \fl_l(\xi', \lambda)} \CM_l(x_N)
	\sigma |\xi'|^2 \hat h(0) \label{un}
\end{align}
for $j=1, \ldots, N-1$.
Set $x_N=0$ in \eqref{un}, then we have 
\begin{equation}\label{un0}
	\hat u_N(0) = \frac{1}{2 \mu \omega_\lambda} \sigma |\xi'|^2 \hat h(0)
	+ \sum^2_{l=1} \frac{(-1)^{l+1} t_1 t_2 (\omega_\lambda^2 +|\xi'|^2) \fa(\xi', \lambda) \fq_l(\xi', \lambda)}
	{2\mu \omega_\lambda s_l \fl_l(\xi', \lambda)} \sigma |\xi'|^2 \hat h(0).
\end{equation}
Substituting \eqref{un0} into the last equation of \eqref{r} yields that
\begin{equation}\label{h}
	\hat h(0) = 
	\fm(\xi', \lambda)\hat \eta(0),
\end{equation}
where 
\begin{align}
	\fm(\xi', \lambda)&=\frac{2\mu \omega_\lambda \kappa^{-1} \fl_1(\xi', \lambda) \fl_2(\xi', \lambda)}{\bM(\xi', \lambda)}, \nonumber \\
	\bM(\xi', \lambda)&= 2\mu \omega_\lambda \lambda  \kappa^{-1}  \fl_1(\xi', \lambda) \fl_2(\xi', \lambda)  \nonumber\\
	& \enskip + \sigma|\xi'|^2 \kappa^{-1} \fl_1(\xi', \lambda) \fl_2(\xi', \lambda) \label{m}\\
	& \enskip
	+  \sigma|\xi'|^2 t_1 t_2 (\omega_\lambda^2 +|\xi'|^2) \fa(\xi', \lambda) 
	(s_2 \fl_2(\xi', \lambda) \fq_1(\xi', \lambda)-s_1 \fl_1(\xi', \lambda) \fq_2(\xi', \lambda)). \nonumber
\end{align}

Inserting \eqref{h} into \eqref{rho}, \eqref{uj}, and \eqref{un}, we have solution formulas for \eqref{r8} as follows:
\begin{equation}\label{sol}
\begin{aligned}
	\rho &= \CF^{-1}_{\xi'}\left[\frac{t_1s_1s_2 (t_1+\omega_\lambda)(\omega_\lambda^2+|\xi'|^2)}{\mu \fl_1(\xi',\lambda)} \sigma|\xi'|^2 \fm(\xi', \lambda) e^{-t_1x_N} \hat \eta(0)\right](x')\\
	&\enskip + \CF^{-1}_{\xi'}\left[\frac{s_1s_2 t_1^2(t_1+\omega_\lambda)(\omega_\lambda^2+|\xi'|^2)}{\mu \fl_1(\xi', \lambda)} \sigma|\xi'|^2 \fm(\xi', \lambda)\CM_0(x_N) \hat \eta(0)\right](x'),\\
	u_j &=\CF^{-1}_{\xi'}\left[\frac{i\xi_j}{2\mu \omega_\lambda^2} \sigma|\xi'|^2 \fm(\xi', \lambda)e^{-\omega_\lambda x_N} \hat \eta(0)\right](x')\\
	&\enskip +
	\sum^2_{l=1} \CF^{-1}_{\xi'}\left[\frac{(-1)^{l+1}i\xi_j t_1 t_2 (\omega_\lambda^2+|\xi'|^2)\fa(\xi', \lambda)\fp_l(\xi', \lambda)}{2\mu \omega_\lambda^2 s_l \fl_l(\xi', \lambda)} \sigma|\xi'|^2 
	\fm(\xi', \lambda)e^{-\omega_\lambda x_N} \hat \eta(0)\right](x')\\
	&\enskip -  
	\sum^2_{l=1} \CF^{-1}_{\xi'}\left[\frac{(-1)^{l+1}i \xi_j t_1t_2 (\omega_\lambda^2 + |\xi'|^2)s_1s_2(t_l+\omega_\lambda)}{\mu s_l \fl_l(\xi, \lambda)} \sigma|\xi'|^2 \fm(\xi', \lambda)\CM_l(x_N)
\hat \eta(0)\right](x'), \\
	u_N &=\CF^{-1}_{\xi'}\left[\frac{1}{2\mu \omega_\lambda} \sigma|\xi'|^2 \fm(\xi', \lambda)e^{-\omega_\lambda x_N}\hat \eta(0)\right](x')\\
	&\enskip + \sum^2_{l=1}\CF^{-1}_{\xi'}\left[\frac{(-1)^{l+1}t_1 t_2 (\omega_\lambda^2 +|\xi'|^2)\fa(\xi', \lambda) \fq_l(\xi', \lambda)}{2\mu \omega_\lambda s_l\fl_l(\xi', \lambda)} \sigma|\xi'|^2 \fm(\xi', \lambda)e^{-\omega_\lambda x_N}\hat \eta(0)\right](x')\\
	&\enskip + 
	\sum^2_{l=1}\CF^{-1}_{\xi'}\left[\frac{(-1)^{l+1}t_1t_2 (\omega_\lambda^2 +|\xi'|^2)s_1s_2t_l(t_l+\omega_\lambda)}{\mu s_l \fl_l(\xi', \lambda)} \sigma|\xi'|^2 
	\fm(\xi', \lambda) \CM_l(x_N)\hat \eta(0)\right](x')
\end{aligned}
\end{equation}
for $j=1, \ldots, N-1$.
We will discuss $\CR$-boundedness using the solution formulas in Sec. \ref{sec:rbdd} below. 
As the end of the subsection, we verify that $\bM(\xi', \lambda)^{-1}$ is well-defined by the following lemma.
\begin{lemma}
	Assume that $\mu$, $\nu$, $\kappa$, and $\sigma$ are positive constants satisfying \eqref{condi}.
	Then $\bM(\xi', \lambda) \neq 0$ for any $(\xi', \lambda) \in (\R^{N-1} \setminus\{0\}) \times (\overline{\C_+} \setminus \{0\})$, 
	where $\overline{\C_+} = \{z \in \C \mid \Re z \ge 0\}$.
\end{lemma}

\begin{proof}
Let us suppose contradiction that $\bM(\xi', \lambda) = 0$ for some $(\xi', \lambda) \in \R^{N-1} \times (\overline{\C_+} \setminus \{0\})$,
then \eqref{r} with $\eta=0$ has a non-trivial solution.
This implies that there exist $h(0) \neq 0$ and $(\rho(x_N), \bu(x_N)) \neq (0, \bf 0)$ for $x_N>0$ sufficiently smooth and decaying exponentially as $x_N \to \infty$
such that 
\begin{align}
	&\lambda \rho(x_N) + \vp(x_N) = 0, \label{e1}\\
	&\lambda u_j(x_N) - \mu(\pd_N^2 - |\xi'|^2) u_j(x_N) -\nu i\xi_j \vp(x_N) - \kappa i\xi_j(\pd_N^2-|\xi'|^2)\rho(x_N)=0, \label{e2}\\
	&\lambda u_N(x_N) - \mu(\pd_N^2 - |\xi'|^2) u_N(x_N) -\nu \pd_N\vp(x_N) - \kappa \pd_N(\pd_N^2-|\xi'|^2)\rho(x_N)=0, \label{e3}\\
	&\mu(i\xi_j u_N(0) + \pd_N u_j(0))=0, \label{e4}\\
	&2\mu \pd_N u_N(0) + (\nu-\mu) \vp(0) + \kappa(\pd_N^2-|\xi'|^2) \rho(0) + \sigma|\xi'|^2 h(0)=0, \label{e5}\\
	&\pd_N \rho(0)=0, \label{e6}\\
	&\lambda h(0) + u_N(0)=0, \label{e7}
\end{align}
where $\vp(x_N)=\sum_{j=1}^{N-1} i\xi_j u_j(x_N) + \pd_N u_N(x_N)$.
Let
$(a, b)=\int^\infty_0 a(x_N) \overline{b(x_N)}\,dx_N$ and $\|a\|=\sqrt{(a, a)}$ for functions $a=a(x_N)$, $b=b(x_N)$ in $\R_+$,
and let $\Phi=\|\pd_N \vp\|^2 + |\xi'|^2\|\vp\|^2$.

First, we prove
\begin{equation}\label{i}
\begin{aligned}
	&\lambda \sum^{N}_{J=1} \|u_J\|^2 + \mu \sum^{N-1}_{j, k=1}\|i\xi_k u_j\|^2 + \mu \left\|\sum^{N-1}_{j =1}i\xi_j u_j\right\|^2 + 2\mu \|\pd_N u_N\|^2\\
	& \quad +\mu \sum^{N-1}_{j =1}\|\pd_N u_j + i\xi_j u_N\|^2 + (\nu-\mu)\|\vp\|^2 + \frac{\kappa}{|\lambda|^2} \bar{\lambda} \Phi + \bar{\lambda} \sigma |\xi'|^2 |h(0)|^2=0.
\end{aligned}
\end{equation}
Employing the same calculation as in the proof of \cite[Lemma 3.4]{S}, 
\eqref{e2} and \eqref{e3} can be written as 
\begin{align}
	&\lambda u_j(x_N) - \mu\sum^{N-1}_{k=1} i\xi_k(i\xi_k u_j(x_N) +i\xi_j u_k(x_N))
	-\mu \pd_N(\pd_N u_j(x_N) + i\xi_j u_N(x_N)) \nonumber\\
	&\quad -(\nu-\mu) i\xi_j \vp(x_N) 
	- \kappa i\xi_j(\pd_N^2-|\xi'|^2)\rho(x_N)=0, \label{e2'}\\
	&\lambda u_N(x_N) - \mu\sum^{N-1}_{k=1} i\xi_k(i\xi_k u_N(x_N) + \pd_N u_k(x_N))
	-2\mu \pd_N^2 u_N(x_N) \nonumber\\
	&\quad -(\nu-\mu) \pd_N\vp(x_N) - \kappa \pd_N(\pd_N^2-|\xi'|^2)\rho(x_N)=0. \label{e3'}
\end{align}
Multiplying \eqref{e2'} and \eqref{e3'} by $\overline{u_j(x_N)}$ and $\overline{u_N(x_N)}$, respectively, integrating with respect to $x_N$,
 and using integration by parts with \eqref{e4} and \eqref{e5}, we have
\begin{align}
	&\lambda \|u_j\|^2 + \mu\sum^{N-1}_{k=1} (i\xi_k u_j +i\xi_j u_k, i\xi_k u_j)
	+\mu (\pd_N u_N + i\xi_j u_N, \pd_N u_j) \nonumber\\
	&\quad +(\nu-\mu) (\vp, i\xi_j u_j) 
	+ \kappa ((\pd_N^2-|\xi'|^2)\rho, i\xi_j u_j)=0, \nonumber\\
	&\lambda \|u_N\|^2 + \mu\sum^{N-1}_{k=1} (i\xi_k u_N + \pd_N u_k, i\xi_k u_N)
	+2\mu \|\pd_N u_N\|^2 \nonumber\\
	&\quad +(\nu-\mu) (\vp, \pd_N u_N) + \kappa ((\pd_N^2-|\xi'|^2)\rho, \pd_N u_N) +\bar  \lambda \sigma |\xi'|^2 |h(0)|^2=0, \label{uN}
\end{align}
where we have used \eqref{e7} to obtain the last term of the left-hand side of \eqref{uN}.
Summing the above equations, applying integration by parts with \eqref{e6}, and using \eqref{e1},
we have
\begin{align*}
	&\lambda \sum^{N}_{J=1} \|u_J\|^2 + \mu \sum^{N-1}_{j, k=1} (i\xi_k u_j +i\xi_j u_k, i\xi_k u_j)
	+\mu \sum^{N-1}_{j=1} (\pd_N u_j + i\xi_j u_N, \pd_N u_j) \\
	& \quad + \mu \sum^{N-1}_{k=1} (i\xi_k u_N + \pd_N u_k, i\xi_k u_N) + 2\mu \|\pd_N u_N\|^2
	+ (\nu-\mu)\|\vp\|^2 + \frac{\kappa}{\lambda} \Phi + \bar \lambda \sigma |\xi'|^2 |h(0)|^2=0. 
\end{align*}
Noting that
\begin{align*}
	&\sum^{N-1}_{j, k=1} (i\xi_k u_j +i\xi_j u_k, i\xi_k u_j) 
	=\sum^{N-1}_{j, k=1}\|i\xi_k u_j\|^2 + \left\|\sum^{N-1}_{j =1}i\xi_j u_j\right\|^2,\\
	&\sum^{N-1}_{j=1} (\pd_N u_j + i\xi_j u_N, \pd_N u_j) + \sum^{N-1}_{k=1} (i\xi_k u_N + \pd_N u_k, i\xi_k u_N)
	= \sum^{N-1}_{j =1}\|\pd_N u_j + i\xi_j u_N\|^2,
\end{align*}
we have \eqref{i}.

Next, we prove the contradiction.
For this purpose, we take the real part and imaginary part of \eqref{i} as follows:
\begin{align}
	&(\Re \lambda) \sum^{N}_{J=1} \|u_J\|^2 + \mu \sum^{N-1}_{j, k=1}\|i\xi_k u_j\|^2 + \mu \left\|\sum^{N-1}_{j =1}i\xi_j u_j\right\|^2 + 2\mu \|\pd_N u_N\|^2\nonumber\\
	& \quad \mu \sum^{N-1}_{j =1}\|\pd_N u_j + i\xi_j u_N\|^2 + (\nu-\mu)\|\vp\|^2 + \frac{\kappa}{|\lambda|^2} (\Re \lambda) \Phi + (\Re \lambda) \sigma|\xi'|^2 |h(0)|^2=0, \label{re}\\
	&(\Im \lambda) \left(\sum^{N}_{J=1} \|u_J\|^2 - \frac{\kappa}{|\lambda|^2} \Phi - \sigma |\xi'|^2 |h(0)|^2\right)=0. \label{im}
\end{align}
Combining \eqref{re} and
\[
	\|\vp\|^2 \le \sum^{N-1}_{j, k=1} \|i \xi_k u_j\|^2 + \left\|\sum^{N-1}_{j =1}i\xi_j u_j\right\|^2 + 2\|\pd_N u_N\|^2,
\]
we have
\begin{equation*}\label{re'}
	(\Re \lambda) \sum^{N}_{J=1} \|u_J\|^2 + \mu \sum^{N-1}_{j =1}\|\pd_N u_j + i\xi_j u_N\|^2 + \nu\|\vp\|^2 
	+ \frac{\kappa}{|\lambda|^2} (\Re \lambda) \Phi + (\Re \lambda) \sigma|\xi'|^2 |h(0)|^2 \le 0.
\end{equation*}
Thus, since $\Re \lambda \ge 0$, we have $\vp=0$, which implies that 
$\rho=0$ for $\lambda \in \overline{\C_+} \setminus \{0\}$ by \eqref{e1}.
Combining $\vp=0$ and \eqref{re}, we have
\[
	0 = \sum^{N-1}_{j, k=1}\|i\xi_k u_j\|^2 = -|\xi'|^2 \sum^{N-1}_{j=1}\|u_j\|^2,
\]
furnishes $u_j=0$ for $j=1, \ldots, N-1$ if $\xi' \neq 0$. 
Since $\sum^{N-1}_{j =1}\|\pd_N u_j + i\xi_j u_N\|^2=0$, we also have 
\[
	0 = \sum^{N-1}_{j =1}\|i\xi_j u_N\|^2 = -|\xi'|^2 \|u_N\|^2, 
\]
then we have $u_N=0$ if $\xi' \neq 0$.
Moreover, $h(0)=0$ follows from \eqref{re} when $\Re \lambda > 0$. 
Note that $\Im \lambda \neq 0$ when $\Re \lambda = 0$, 
then \eqref{im} implies that $h(0)=0$ 
for $(\xi', \lambda) \in (\R^{N-1} \setminus\{0\}) \times (\overline{\C_+} \setminus \{0\})$
because $\Phi=0$ follows from $\vp=0$.

Summing up, we have $\rho=0$, $\bu=\bf 0$, and $h(0)=0$, which contradicts \eqref{r} with $\eta=0$ has a non-trivial solution when $\xi' \in \R^{N-1} \setminus\{0\}$, $\lambda \in \overline{\C_+} \setminus \{0\}$.
\end{proof}

\subsection{Preliminary results}
In this subsection, we prepare the estimates and recall some results to show the existence of the $\CR$-bounded operator families. 
First, we introduce the definition of the class of symbols.

\begin{definition}
Let $\Lambda \subset \C$, and let $m(\xi', \lambda)$ be a function defined on $\Lambda$
that is infinitely many times differentiable with respect to $\xi'$ and analytic with respect to $\lambda$.

$\thetag1$
$m(\xi', \lambda)$ is called multiplier of order $s$ with type $1$, denoted by $m \in \bM^1_s (\Lambda)$,
if there exists a real number $s$ such that for any multi-index $\alpha'=(\alpha_1, \ldots, \alpha_{N-1}) \in \bN^{N-1}_0$
and $(\xi',\lambda) \in (\R^{N-1}\setminus \{0\}) \times \Lambda$
\[
	|\pd^{\alpha'}_{\xi'}((\tau \pd_\tau)^n m(\xi', \lambda))| \le C(|\lambda|^{1/2}+|\xi'|)^{s-|\alpha'|} \quad (n=0, 1)
\]
with some constant $C=C_{s, \alpha', \Lambda}$.

$\thetag2$
$m(\xi', \lambda)$ is called multiplier of order $s$ with type $2$, denoted by $m \in \bM^2_s (\Lambda)$,
if there exists a real number $s$ such that for any multi-index $\alpha'=(\alpha_1, \ldots, \alpha_{N-1}) \in \bN^{N-1}_0$
and $(\xi',\lambda) \in (\R^{N-1}\setminus \{0\}) \times \Lambda$
\[
	|\pd^{\alpha'}_{\xi'}((\tau \pd_\tau)^n m(\xi', \lambda))| \le C(|\lambda|^{1/2}+|\xi'|)^s|\xi'|^{-|\alpha'|} \quad (n=0, 1)
\]
with some constant $C=C_{s, \alpha', \Lambda}$.

\end{definition}

\begin{remark}\label{remarkm}
Let $s_1, s_2 \in \R$, $\lambda_0 \ge 0$, and $\epsilon \in (0, \pi/2)$.
%
For $m_j \in \bM^1_{s_j}(\Lambda)$~$(j=1, 2)$, we have $m_1 m_2 \in \bM^1_{s_1+s_2}(\Lambda)$.
%
%
%
\end{remark}
In the next section, we will use the following corollary and lemmas to show $\CR$-boundedness for operator families.
Lemma \ref{lem:t1} was proved by \cite[Lemma 4.8]{S} and
Lemma \ref{lem:t3} was proved by \cite[Lemma 3.5.16]{Shi2020}. 
\begin{lemma}\label{lem:t1} Let $1 < q < \infty$, $\lambda_0 \ge 0$, $\epsilon_1 \in (0, \pi/2)$, and $\epsilon_2 \in (\tilde \epsilon_*, \pi/2)$ for $\tilde \epsilon_*$ given in \eqref{angle}. 
For 
\[
	k (\xi', \lambda) \in \bM^1_1(\Sigma_{\epsilon_1, \lambda_0}),
	\quad
	l (\xi', \lambda) \in \bM^1_1(\Sigma_{\epsilon_2, \lambda_0}),
	\quad
	m (\xi', \lambda) \in \bM^1_2(\Sigma_{\epsilon_2, \lambda_0}),
\] 
we define operators $K(\lambda)$, $L(\lambda)$, and $M_j(\lambda)$~$(j=0, 1, 2)$ by
\begin{align*}
	[K(\lambda)f](x) & = \int^\infty_0 \CF^{-1}_{\xi'}[k (\xi', \lambda)
	e^{-\omega_\lambda (x_N+y_N)} \hat f(\xi', y_N)](x')\, dy_N &(\lambda \in \Sigma_{\epsilon_1, \lambda_0}),\\
	[L(\lambda)f](x) & = \int^\infty_0 \CF^{-1}_{\xi'}[l (\xi', \lambda)
	e^{-t_1 (x_N+y_N)} \hat f(\xi', y_N)](x')\, dy_N &(\lambda \in \Sigma_{\epsilon_2, \lambda_0}),\\
	[M_j(\lambda)f](x) & = \int^\infty_0 \CF^{-1}_{\xi'}[m (\xi', \lambda)
	\CM_j(x_N+y_N) \hat f(\xi', y_N)](x')\, dy_N &(\lambda \in \Sigma_{\epsilon_2, \lambda_0}).
\end{align*}
Then we have 
\begin{align*}
	&\CR_{\CL(L_q(\R^N_+))}
	(\{(\tau \pd_\tau)^n K(\lambda) \mid 
	\lambda \in \Sigma_{\epsilon_1, \lambda_0} \}) \leq k,\\
	&\CR_{\CL(L_q(\R^N_+))}
	(\{(\tau \pd_\tau)^n L(\lambda) \mid 
	\lambda \in \Sigma_{\epsilon_2, \lambda_0} \}) \leq l,\\
	&\CR_{\CL(L_q(\R^N_+))}
	(\{(\tau \pd_\tau)^n M_j(\lambda) \mid 
	\lambda \in \Sigma_{\epsilon_2, \lambda_0} \}) \leq m_j
\end{align*}
for $n=0, 1$, where $k$ is some constant depending on $\epsilon_1, \lambda_0, \mu, \nu, \kappa, N, q$
and $l, m_j$ are some constants depending on $\epsilon_2, \lambda_0, \mu, \nu, \kappa, N, q$.
\end{lemma}

Lemma \ref{lem:t1} implies that the following corollary.

\begin{corollary}\label{cor:t1}
Let $1 < q < \infty$, $\lambda_0 > 0$, $\epsilon_1 \in (0, \pi/2)$, and $\epsilon_2 \in (\tilde \epsilon_*, \pi/2)$ for $\tilde \epsilon_*$ given in \eqref{angle}. 

$\thetag1$
For 
\[
	k (\xi', \lambda) \in \bM^1_{-1}(\Sigma_{\epsilon_1, \lambda_0}),
	\quad
	m (\xi', \lambda) \in \bM^1_0(\Sigma_{\epsilon_2, \lambda_0}),
\] 
we define operators $K(\lambda)$ and $M_l(\lambda)$~$(l=1, 2)$ by
\begin{align*}
	[K(\lambda)f](x) & = \int^\infty_0 \CF^{-1}_{\xi'}[k (\xi', \lambda)
	e^{-\omega_\lambda (x_N+y_N)} \hat f(\xi', y_N)](x')\, dy_N &(\lambda \in \Sigma_{\epsilon_1, \lambda_0}),\\
	[M_l(\lambda)f](x) & = \int^\infty_0 \CF^{-1}_{\xi'}[m (\xi', \lambda)
	\CM_l(x_N+y_N) \hat f(\xi', y_N)](x')\, dy_N &(\lambda \in \Sigma_{\epsilon_2, \lambda_0}).
\end{align*}
Then we have 
\begin{align*}
	&\CR_{\CL(L_q(\R^N_+), H^{2-j}_q(\R^N_+))}
	(\{(\tau \pd_\tau)^n \lambda^{j/2} K(\lambda) \mid 
	\lambda \in \Sigma_{\epsilon_1, \lambda_0} \}) \leq k,\\
	&\CR_{\CL(L_q(\R^N_+), H^{2-j}_q(\R^N_+))}
	(\{(\tau \pd_\tau)^n \lambda^{j/2} M_l(\lambda) \mid 
	\lambda \in \Sigma_{\epsilon_2, \lambda_0} \}) \leq m_l
\end{align*}
for  $j=0, 1, 2$, $n=0, 1$, where $k$ is some constant depending on $\epsilon_1, \lambda_0, \mu, \nu, \kappa, N, q$
and $m_l$ are some constants depending on $\epsilon_2, \lambda_0, \mu, \nu, \kappa, N, q$.

$\thetag2$
For 
\[
	l (\xi', \lambda) \in \bM^1_{-2}(\Sigma_{\epsilon_2, \lambda_0}),
	\quad
	m (\xi', \lambda) \in \bM^1_{-1}(\Sigma_{\epsilon_2, \lambda_0}),
\] 
we define operators $L(\lambda)$ and $M_0(\lambda)$ by
\begin{align*}
	[L(\lambda)f](x) & = \int^\infty_0 \CF^{-1}_{\xi'}[l (\xi', \lambda)
	e^{-t_1 (x_N+y_N)} \hat f(\xi', y_N)](x')\, dy_N &(\lambda \in \Sigma_{\epsilon_2, \lambda_0}),\\
	[M_0(\lambda)f](x) & = \int^\infty_0 \CF^{-1}_{\xi'}[m (\xi', \lambda)
	\CM_0(x_N+y_N) \hat f(\xi', y_N)](x')\, dy_N &(\lambda \in \Sigma_{\epsilon_2, \lambda_0}).
\end{align*}
Then we have 
\begin{align*}
	&\CR_{\CL(L_q(\R^N_+), H^{3-j}_q(\R^N_+))}
	(\{(\tau \pd_\tau)^n \lambda^{j/2} L(\lambda) \mid 
	\lambda \in \Sigma_{\epsilon_2, \lambda_0} \}) \leq l,\\
	&\CR_{\CL(L_q(\R^N_+), H^{3-j}_q(\R^N_+))}
	(\{(\tau \pd_\tau)^n \lambda^{j/2} M_0(\lambda) \mid 
	\lambda \in \Sigma_{\epsilon_2, \lambda_0} \}) \leq m_0
\end{align*}
for $j=0, 1, 2$, $n=0, 1$, where $l$ and $m_0$ are some constants depending on $\epsilon_2, \lambda_0, \mu, \nu, \kappa, N, q$.

\end{corollary}

\begin{lemma}\label{lem:t3} Let $1 < q < \infty$ and $j=1, 2$, and 
let $\Lambda$ be a domain in $\C$. Let $\varphi$ and $\psi$ be two $C^\infty_0((-2, 2))$ functions. 
For $m(\xi', \lambda) \in \bM^2_0(\Lambda)$,
we define operators $\Phi_j(\lambda)$ by
\begin{align*}
	[\Phi_1(\lambda)f](x) & = \varphi(x_N) \int^\infty_0\CF^{-1}_{\xi'}[m(\xi', \lambda)
	e^{-|\xi'| (x_N+y_N)} \psi(y_N) \hat f(\xi', y_N)](x')\,dy_N, \\
	[\Phi_2(\lambda)f](x) & = \varphi(x_N) \int^\infty_0\CF^{-1}_{\xi'}[m(\xi', \lambda)
	|\xi'|e^{-|\xi'| (x_N+y_N)} \psi(y_N) \hat f(\xi', y_N)](x')\,dy_N.
\end{align*}
Then we have 
\begin{align*}
	\CR_{\CL(L_q(\R^N_+))}
	(\{(\tau \pd_\tau)^n \Phi_j(\lambda) \mid 
	\lambda \in \Lambda \}) \leq r_j,
\end{align*}
for $n=0, 1$, and some constant $r_j$ depending on $\Lambda, \varphi, \psi, N, q$.
\end{lemma}

Next, we recall the class of several symbols proved by Shibata and Shimizu \cite[Lemma 5.2]{SS} and Saito \cite[Corollary 4.7]{S}, in order to apply the above lemmas,
where these symbols
are defined in \eqref{notation}.
\begin{lemma}\label{est:fl}
Let $s \in \R$. The following assertions hold true.

$\thetag1$
Let $\epsilon \in (0, \pi/2)$. Then we have 
\begin{equation*}\label{omega}
	\omega_\lambda^s \in \bM^1_s(\Sigma_{\epsilon, 0}).
\end{equation*}

$\thetag2$
Let $\epsilon \in (\tilde \epsilon_*, \pi/2)$. Then we have 
\begin{equation*}\label{tpqa}
	t^s_j \in \bM^1_s(\Sigma_{\epsilon, 0}), \quad
	\fp_j (\xi', \lambda), \fq_j  (\xi', \lambda) \in \bM^1_1(\Sigma_{\epsilon, 0}), \quad
	\fa (\xi', \lambda) \in \bM^1_1(\Sigma_{\epsilon, 0})
\end{equation*}
for $j=1, 2$.

$\thetag3$
There is a constant $\epsilon_* \in (\tilde \epsilon_*, \pi/2)$ such that for any $\epsilon \in (\epsilon_*, \pi/2)$
\begin{equation}\label{fl}
	\fl_j(\xi', \lambda)^s \in \bM^1_{6s}(\Sigma_{\epsilon, 0})
\end{equation}
for $j=1, 2$. 
\end{lemma} 

The lower bound of $\bM(\xi', \lambda)$ is obtained by Lemma \ref{est:fl}, where 
$\bM(\xi', \lambda)$ is defined in \eqref{m}.
\begin{lemma}\label{lem:em}
Let $\epsilon \in (\epsilon_*, \pi/2)$ for the same constant $\epsilon_*$ as in \eqref{fl}.
There exists a positive constant $C$ and $\lambda_0 >0$ such that 
\begin{equation}\label{em}
	|\bM(\xi', \lambda)| \ge C(|\lambda|+|\xi|)(|\lambda|^{1/2} + |\xi'|)^{13}
\end{equation}
for any $(\xi', \lambda) \in (\R^{N-1}\setminus\{0\}) \times \Sigma_{\epsilon, \lambda_0}$.
Here $\lambda_0$ and $C$ depend on $\epsilon$, $\mu$, $\nu$, $\kappa$, and 
$\sigma$.
\end{lemma} 

\begin{proof}
First, we consider the case that $|\xi'|/|\lambda| \le r$ for $(\xi', \lambda) \in \R^{N-1} \times \Sigma_{\epsilon, 0}$
for some sufficiently small positive number $r$ 
determined below. 
Lemma \ref{est:fl} implies that there exist constants $C$ and $c$ such that
\begin{equation*}
\begin{aligned}
	|\bM(\xi', \lambda)| 
	&\ge 2\mu |\omega_\lambda||\lambda|\kappa^{-1} |\fl_1(\xi', \lambda)||\fl_2(\xi', \lambda)|\\
	&\enskip - \sigma |\xi'|^2 \kappa^{-1} |\fl_1(\xi', \lambda)||\fl_2(\xi', \lambda)|\\
	&\enskip -\sigma |\xi'|^2 |t_1||t_2|(|\omega_\lambda^2|+|\xi|^2)|\fa(\xi', \lambda)|
	|s_2||\fl_2(\xi', \lambda)||\fq_1(\xi', \lambda)|\\
	&\enskip -\sigma |\xi'|^2 |t_1||t_2|(|\omega_\lambda^2|+|\xi|^2)|\fa(\xi', \lambda)|
	|s_1||\fl_1(\xi', \lambda)||\fq_2(\xi', \lambda)|\\
	&\ge c|\lambda|(|\lambda|^{1/2}+|\xi'|)^{13}\{1-C|\lambda|^{-1}(|\lambda|^{1/2}+|\xi'|)^{-1}|\xi'|^2\}\\
	&\ge c|\lambda|(|\lambda|^{1/2}+|\xi'|)^{13}(1-Cr)
\end{aligned}
\end{equation*}
provided $|\xi'| \le r|\lambda|$, 
then we have chosen $r$ sufficiently small that $1-Cr>1/2$, which furnishes \eqref{em}
for any $(\xi', \lambda) \in (\R^{N-1}\setminus\{0\}) \times \Sigma_{\epsilon, 0}$.

Next, we consider the case that $|\xi'|/|\lambda| \ge r$ for $(\xi', \lambda) \in \R^{N-1} \times \Sigma_{\epsilon, \lambda_0}$. 
Note that $|\xi'| \ge r|\lambda|^{1/2} \lambda_0^{1/2}$, 
then choosing $\lambda_0$ large enough, we have
\[
	t_j = |\xi'| (1+O(\lambda_0^{-1/2})), \quad
	\omega_\lambda = |\xi'| (1+O(\lambda_0^{-1/2})).
\]
Thus, we also have
\begin{align*}
	\fl_j (\xi', \lambda) &= 8\mu^{-1} |\xi'|^6 (1+O(\lambda_0^{-1/2})),\\
	\fa(\xi', \lambda) &= \frac{2\kappa^{-1}}{s_2-s_1}|\xi'| (1+ O(\lambda_0^{-1/2})),\\
	\fq_j(\xi', \lambda) &= 2s_j |\xi'| (1+O(\lambda_0^{-1/2})),
\end{align*}
furnishes that there exists a positive constant $K$ such that
\[
	\left|\frac{\bM(\xi', \lambda)-64\mu^{-2}\kappa^{-1}|\xi'|^{13}(2\mu \lambda  + \sigma|\xi'|)}{64\mu^{-2}\kappa^{-1}|\xi'|^{13}(2\mu \lambda  + \sigma|\xi'|)}\right|
	\le K \lambda_0^{-1/2}
\]
for $|\lambda| \ge \lambda_0$ and sufficient large $\lambda_0$.
Therefore, we have
\begin{align*}
|\bM(\xi', \lambda)| 
	&\ge 64\mu^{-2}\kappa^{-1} |\xi'|^{13} \left|2\mu \lambda  + \sigma|\xi'|\right|
	-64\mu^{-2}\kappa^{-1}|\xi'|^{13}(2\mu |\lambda| + \sigma|\xi'|) K \lambda_0^{-1/2} \\
	&\ge 64\sin (\epsilon/2) \mu^{-2}\kappa^{-1}|\xi'|^{13}(2\mu |\lambda| + \sigma|\xi'|)
	-64\mu^{-2}\kappa^{-1}|\xi'|^{13}(2\mu |\lambda| + \sigma|\xi'|) K \lambda_0^{-1/2}.
\end{align*}
Choosing $\lambda_0 >0$ so large that $\sin (\epsilon/2)/2 - K \lambda_0^{-1/2} \ge 0$, we have 
\[
	|\bM(\xi', \lambda)| \ge 32(\sin \epsilon/2)\mu^{-2}\kappa^{-1}|\xi'|^{13}(2\mu |\lambda| + \sigma|\xi'|).
\]
Then \eqref{em} holds true provided by $|\xi'| \ge r |\lambda|$. 
This completes the proof of Lemma \ref{lem:em}.
\end{proof}

Note that $|\xi'|^2 \in \bM^1_2(\Sigma_{\epsilon, 0})$.
The following corollary is proved by Bell's formula, Lemma \ref{est:fl}, and Lemma \ref{lem:em}.

\begin{corollary}\label{mm}
Let $\epsilon \in (\epsilon_*, \pi/2)$ for the same constant $\epsilon_*$ as in \eqref{fl}.
Then there exist some constants $\lambda_0 >0$ and $C$ such that for any multi-index 
$\alpha' \in \N^{N-1}_0$ and 
$(\xi', \lambda) \in (\R^{N-1}\setminus\{0\}) \times \Sigma_{\epsilon, \lambda_0}$
\[
	|\pd^{\alpha'}_{\xi'}((\tau \pd_\tau)^n \bM(\xi', \lambda)^{-1})| 
	\le C(|\lambda|+|\xi'|)^{-1}(|\lambda|^{1/2}+|\xi'|)^{-13-|\alpha'|},
\]
where $C$ depend on $\epsilon$, $\alpha'$, $\mu$, $\nu$, $\kappa$, and $\sigma$.
\end{corollary}
Lemma \ref{est:fl} and Corollary \ref{mm} imply that the following corollary. 
\begin{corollary}\label{fm}
Let $\epsilon \in (\epsilon_*, \pi/2)$ for the same constant $\epsilon_*$ as in \eqref{fl}.
Then there exist some constants $\lambda_0 \ge 1$ and $C$ such that for any multi-index 
$\alpha' \in \N^{N-1}_0$ and 
$(\xi', \lambda) \in (\R^{N-1}\setminus\{0\}) \times \Sigma_{\epsilon, \lambda_0}$
\[
	|\pd^{\alpha'}_{\xi'}((\tau \pd_\tau)^n \fm(\xi', \lambda))| 
	\le C(|\lambda|+|\xi'|)^{-1}(|\lambda|^{1/2}+|\xi'|)^{-|\alpha'|},
\]
where $C$ depend on $\epsilon$, $\alpha'$, $\mu$, $\nu$, $\kappa$, and $\sigma$.
\end{corollary}

\section{The proof of Proposition \ref{with}}\label{sec:rbdd}
In this section, we prove \eqref{est:rbdd}, which is the $\CR$-boundedness for the operators $\CA_2(\lambda)$,
$\CB_2(\lambda)$, and $\CC_2(\lambda)$ arising from solution formulas \eqref{sol} and \eqref{h} for \eqref{r8}.
Note that \eqref{family} in Proposition \ref{with} $\thetag1$ is obtained by \eqref{est:rbdd} due to Definition \ref{dfn2}.
For $Z \in \{\omega_\lambda, t_1, |\xi'|\}$, let us review the Volevich trick, which is the following equalities:
\begin{equation}\label{vt}
\begin{aligned}
	e^{-Z x_N} \hat \eta(\xi', 0) &= \int^\infty_0 Ze^{-Z(x_N + y_N)} \hat \eta(\xi, y_N)\,dy_N
	-\int^\infty_0 e^{-Z(x_N + y_N)} \widehat{\pd_N \eta}(\xi', y_N)\,dy_N,\\
	\CM_l(x_N) \hat \eta(\xi', 0) &= \int^\infty_0 t_l \CM_l(x_N+y_N) \hat \eta(\xi, y_N)\,dy_N
	+\int^\infty_0 \fr_l e^{-\omega_\lambda(x_N+y_N)} \hat \eta(\xi', y_N)\,dy_N\\
	&\enskip -\int^\infty_0 \CM_l(x_N+y_N) \widehat{\pd_N \eta}(\xi', y_N)\,dy_N,\\
	\CM_0(x_N) \hat \eta(\xi', 0) &= \int^\infty_0 t_2 \CM_0(x_N+y_N) \hat \eta(\xi, y_N)\,dy_N
	+\int^\infty_0 e^{-t_1(x_N+y_N)} \hat \eta(\xi', y_N)\,dy_N\\
	&\enskip -\int^\infty_0 \CM_0(x_N+y_N) \widehat{\pd_N \eta}(\xi', y_N)\,dy_N,
\end{aligned}
\end{equation}
where $l=1, 2$ and we have used 
\begin{align*}
	\pd_N \CM_l(z_N) &= -t_l\CM_l(z_N) - \fr_l(\xi', \lambda) e^{-\omega_\lambda z_N},
	\quad
	\fr_l(\xi', \lambda)=\frac{(s_l -\mu^{-1})(t_2+t_1)}{(s_2-s_1)(t_l + \omega_\lambda)},\\
	\pd_N \CM_0(z_N) &= -t_2 \CM_0(z_N)-e^{-t_1z_N}
\end{align*} 
for $z_N>0$, which is obtained by \eqref{notation}.

\subsection{The $\CR$-boundedness for $\CA_2(\lambda)$}\label{a2}
Recall the solution formula \eqref{sol}:
\[
\rho = \CF^{-1}_{\xi'}\left[\vartheta_1 |\xi'|^2 \fm(\xi', \lambda) e^{-t_1x_N} \hat \eta(0)\right](x')
	+ \CF^{-1}_{\xi'}\left[\vartheta_2 |\xi'|^2 \fm(\xi', \lambda)\CM_0(x_N) \hat \eta(0)\right](x'),
\]
with 
\begin{align*}
\vartheta_1(\xi', \lambda)&= \frac{t_1s_1s_2 (t_1+\omega_\lambda)(\omega_\lambda^2+|\xi'|^2)}{\mu \fl_1(\xi',\lambda)},\\
\vartheta_2(\xi', \lambda)&= \frac{s_1s_2 t_1^2(t_1+\omega_\lambda)(\omega_\lambda^2+|\xi'|^2)}{\mu \fl_1(\xi', \lambda)}.
\end{align*}
In view of \eqref{vt} and $|\xi'|^2=-\sum^{N-1}_{k=1}(i\xi_k)^2$, we set
\begin{align*}
	\rho=\CA_2(\lambda) \eta
	=\CA_2^1(\lambda)(-\Delta' \eta) + \CA_2^2(\lambda)(\nabla' \pd_N \eta), 
\end{align*}
where
\begin{align*}
	\CA_{2}^1(\lambda)(-\Delta' \eta)
	&=\int_0^\infty \CF^{-1}_{\xi'}[\vartheta_1(\xi', \lambda)\fm(\xi', \lambda) t_1 e^{-t_{1}(x_N+y_N)}\widehat{(-\Delta') \eta}(\xi', y_N)](x')\,dy_N\\
	&\enskip + \int^\infty_0 \CF^{-1}_{\xi'}[\vartheta_2 (\xi', \lambda)\fm(\xi', \lambda) t_2 \CM_0 (x_N+y_N) \widehat{(-\Delta') \eta} (\xi', y_N)](x')\,dy_N\\
	&\enskip + \int^\infty_0 \CF^{-1}_{\xi'}[\vartheta_2 (\xi', \lambda)\fm(\xi', \lambda) e^{-t_1 (x_N+y_N)}\widehat{(-\Delta') \eta}(\xi', y_N)](x')\,dy_N,\\
	\CA_{2}^2(\lambda)(\nabla' \pd_N \eta)
	&=\sum_{k=1}^{N-1} \int^\infty_0 \CF^{-1}_{\xi'}[\vartheta_1 (\xi', \lambda)\fm(\xi', \lambda) i\xi_k 
	e^{-t_1 (x_N+y_N)}\widehat{\pd_k \pd_N \eta}(\xi', y_N)](x')\,dy_N\\
	&\enskip + \sum_{k=1}^{N-1} \int^\infty_0 \CF^{-1}_{\xi'}[\vartheta_2 (\xi', \lambda)\fm(\xi', \lambda) i\xi_k 	
	\CM_0 (x_N+y_N) \widehat{\pd_k \pd_N \eta}(\xi', y_N)](x')\,dy_N.
\end{align*}
Lemma \ref{est:fl} and Remark \ref{remarkm} imply that for any $\epsilon \in (\epsilon_*, \pi/2)$ for the same constant $\epsilon_*$ as in \eqref{fl}
\[
	\vartheta_1(\xi', \lambda) \in \bM^1_{-2}(\Sigma_{\epsilon, 0}), \quad \vartheta_2(\xi', \lambda) \in \bM^1_{-1}(\Sigma_{\epsilon, 0}).
\]
Then by Lemma \ref{est:fl} and Corollary \ref{fm}, together with Leibniz's rule, 
there exists $\lambda_0 \ge 1$ such that 
\begin{align*}
	\vartheta_2(\xi', \lambda) \fm(\xi', \lambda) t_2, \enskip  
	\vartheta_2 (\xi', \lambda)\fm(\xi', \lambda) i\xi_k &\in \bM^1_{-1}(\Sigma_{\epsilon, \lambda_0}),\\
	\vartheta_1(\xi', \lambda) \fm(\xi', \lambda) t_1, \enskip
	\vartheta_1 (\xi', \lambda)\fm(\xi', \lambda) i\xi_k, \enskip 
	\vartheta_2 (\xi', \lambda)\fm(\xi', \lambda), &\in \bM^1_{-2}(\Sigma_{\epsilon, \lambda_0})
\end{align*}
for $k=1, \ldots, N-1$.
By using Corollary \ref{cor:t1} $\thetag2$, it implies that $\CA_2(\lambda)$ satisfies \eqref{est:rbdd}.

\subsection{The $\CR$-boundedness for $\CB_2(\lambda)$}
The solution formulas \eqref{sol} for \eqref{r8} implies that 
\[
	u_J = \CF^{-1}_{\xi'} [n_{J1}(\xi', \lambda) \fm(\xi', \lambda) 
	e^{-\omega_\lambda x_N}|\xi'|^2 \hat \eta(0)](x')
	+ \sum^2_{l=1} \CF^{-1}_{\xi'} [n_{J2}^l(\xi', \lambda) \fm(\xi', \lambda) 
	\CM_l(x_N) |\xi'|^2 \hat \eta(0)](x'),
\]
where 
\begin{align*}
	n_{j1}(\xi', \lambda)&= \frac{i\xi_j}{2\mu \omega_\lambda^2}
	+\sum^2_{l=1} \frac{(-1)^{l+1}i\xi_j t_1 t_2 (\omega_\lambda^2+|\xi'|^2)\fa(\xi', \lambda)\fp_l(\xi', \lambda)}{2\mu \omega_\lambda^2 s_l \fl_l(\xi', \lambda)},\\
	n_{j2}^l(\xi', \lambda)&=
	\frac{(-1)^{l}i \xi_j t_1t_2 (\omega_\lambda^2 + |\xi'|^2)s_1s_2(t_l+\omega_\lambda)}{\mu s_l \fl_l(\xi, \lambda)},\\ 
	n_{N1}(\xi', \lambda)&= \frac{1}{2\mu \omega_\lambda} 	
	+ \sum^2_{l=1} \frac{(-1)^{l+1}t_1 t_2 (\omega_\lambda^2 +|\xi'|^2)\fa(\xi', \lambda) \fq_l(\xi', \lambda)}{2\mu \omega_\lambda s_l\fl_l(\xi', \lambda)}, \\
	n_{N2}^l(\xi', \lambda)&= \frac{(-1)^{l+1}t_1t_2 (\omega_\lambda^2 +|\xi'|^2)s_1s_2t_l(t_l+\omega_\lambda)}{\mu s_l \fl_l(\xi', \lambda)} 
\end{align*}
for $J=1, \ldots, N$, $j=1, \ldots, N-1$.
Repeat the same method as in subsection \ref{a2}.
Set
\begin{align*}
	u_J=\CB_{2J}(\lambda) \eta
	=\CB_{2J}^1(\lambda)(-\Delta' \eta) + \CB_{2J}^2(\lambda)(\nabla' \pd_N \eta), 
\end{align*}
where
\begin{align*}
	\CB_{2J}^1(\lambda)(-\Delta' \eta)
	&=\int^\infty_0 \CF^{-1}_{\xi'}[n_{J1} \fm(\xi', \lambda) \omega_\lambda e^{-\omega_\lambda (x_N+y_N)}\widehat{(-\Delta') \eta}(\xi', y_N)](x')\,dy_N\\
	&\enskip +\sum_{l=1}^2 \int^\infty_0 \CF^{-1}_{\xi'}[n_{J2}^l \fm(\xi', \lambda) t_l \CM_l (x_N+y_N) \widehat{(-\Delta') \eta}(\xi', y_N)](x')\,dy_N\\
	&\enskip +\sum_{l=1}^2 \int^\infty_0 \CF^{-1}_{\xi'}[n_{J2}^l \fm(\xi', \lambda) \fr_l(\xi, \lambda) e^{-\omega_\lambda (x_N+y_N)}\widehat{(-\Delta') \eta}(\xi', y_N)](x')\,dy_N,\\
	\CB_{2J}^2(\lambda)(\nabla' \pd_N \eta)
	&=\sum_{k=1}^{N-1} \int^\infty_0 \CF^{-1}_{\xi'}[n_{J1} \fm(\xi', \lambda) i\xi_k 
	e^{-\omega_\lambda (x_N+y_N)}\widehat{\pd_k \pd_N \eta}(\xi', y_N)](x')\,dy_N\\
	&\enskip + \sum_{k=1}^{N-1} \sum_{l=1}^2 \int^\infty_0 \CF^{-1}_{\xi'}[n_{J2}^l \fm(\xi', \lambda) i\xi_k 	
	\CM_l (x_N+y_N) \widehat{\pd_k \pd_N \eta}(\xi', y_N)](x')\,dy_N
\end{align*}
for $J=1, \ldots, N$.
Lemma \ref{est:fl} and Remark \ref{remarkm} imply that for any $\epsilon \in (\epsilon_*, \pi/2)$ for the same constant $\epsilon_*$ as in \eqref{fl}
\[
	n_{J1}(\xi', \lambda) \in \bM^1_{-1}(\Sigma_{\epsilon, 0}), \quad n_{J2}^l(\xi', \lambda), \enskip \fr_l(\xi', \lambda) \in \bM^1_0(\Sigma_{\epsilon, 0})
\]
for $J=0, \ldots, N$, $l=1, 2$. 
Then by Lemma \ref{est:fl} and Corollary \ref{fm}, together with Leibniz's rule, 
there exists $\lambda_0 \ge 1$ such that 
\begin{align*}
	n_{J1}(\xi', \lambda) \fm(\xi', \lambda) \omega_\lambda, \enskip 
	n_{J2}^l(\xi', \lambda) \fm(\xi', \lambda) \fr_l(\xi, \lambda), \enskip
	n_{J1} \fm(\xi', \lambda) i\xi_k &\in \bM^1_{-1}(\Sigma_{\epsilon, \lambda_0})\\
	n_{J2}^l(\xi', \lambda) \fm(\xi', \lambda) t_l, \enskip 
	n_{J2}^l(\xi', \lambda) \fm(\xi', \lambda) i\xi_k &\in \bM^1_0(\Sigma_{\epsilon, \lambda_0})
\end{align*}
for 
$J=0, \ldots, N$, $k=0, \ldots, N-1$, and $l=1, 2$. 
Thus Corollary \ref{cor:t1} $\thetag1$ implies that $\CB_2(\lambda)$ satisfies \eqref{est:rbdd}.

\subsection{The $\CR$-boundedness for $\CC_2(\lambda)$}
In view of \eqref{h}, we set $h$ as follows:
\begin{equation}\label{eh}
	h(x) = (\CC_2(\lambda)\eta)(x) = \vp(x_N) \CF^{-1}_{\xi'}[\fm(\xi', \lambda)e^{-|\xi'|x_N} \hat \eta(\xi', 0)](x')
\end{equation}
for $x=(x', x_N) \in \R^N_+$, where $\vp(x_N)\in C^\infty_0(\R)$ equals to $1$ for $x_N \in (-1, 1)$ and $0$
for $x_N \notin [-2, 2]$.
Applying the Volevich trick \eqref{vt} to \eqref{eh}, we have
\[
	h(x) = (\CC_2(\lambda)\eta)(x) = (\CC_2^1(\lambda)\eta)(x)+(\CC_2^2(\lambda)\eta)(x),
\]
where
\begin{align*}
	(\CC_2^1(\lambda)\eta)(x)&=\vp(x_N) \int^\infty_0 \CF^{-1}_{\xi'}[\fm(\xi', \lambda)|\xi'| e^{-|\xi'|(x_N+y_N)} \vp(y_N) \hat \eta(\xi', y_N)](x')\,dy_N,\\
	(\CC_2^2(\lambda)\eta)(x)&=-\vp(x_N) \int^\infty_0 \CF^{-1}_{\xi'}[\fm(\xi', \lambda) e^{-|\xi'|(x_N+y_N)} \pd_N(\vp(y_N) \hat \eta(\xi', y_N))](x')\,dy_N.
\end{align*}
First, we consider the $\CR$-boundedness for $\lambda^j\CC_2(\lambda)$ for $j=0, 1$.
Corollary \ref{fm} implies that for any $\epsilon \in (\epsilon_*, \pi/2)$ for the same constant $\epsilon_*$ as in \eqref{fl},
there exists $\lambda_0 \ge 1$ such that
$\lambda^j \fm(\xi', \lambda) \in \bM^1_0(\Sigma_{\epsilon, \lambda_0}) \subset \bM^2_0(\Sigma_{\epsilon, \lambda_0})$ for $j=0, 1$, then Lemma \ref{lem:t3} gives us
there exists a positive constant $r$ such that
\begin{equation} \label{rbddh1}
\begin{aligned}
&\CR_{\CL(L_q(\R^N_+))}
(\{(\tau \pd_\tau)^n \lambda^j \CC_2^1 (\lambda) \mid 
\lambda \in \Sigma_{\epsilon, \lambda_0}\}) 
\leq r,\\
&\CR_{\CL(H^1_q(\R^N_+), L_q(\R^N_+))}
(\{(\tau \pd_\tau)^n \lambda^j \CC_2^2 (\lambda) \mid 
\lambda \in \Sigma_{\epsilon, \lambda_0}\}) 
\leq r
\end{aligned}
\end{equation}
for $n = 0, 1$, $j=0, 1$.

Next, we consider the $\CR$-boundedness for $\lambda^j \pd_{x'}^{\alpha'} \pd_N^k \CC_2(\lambda)$ for $j=0, 1$ and $1\le |\alpha'|+k \le 3-j$.
In this case, we use
\[
	1=\frac{1+|\xi'|^2}{1+|\xi'|^2}
	=\frac{1}{1+|\xi'|^2}-\sum^{N-1}_{l=1} \frac{i\xi_l}{1+|\xi'|^2}i\xi_l.
\]
For $j=0, 1$, $|\alpha'|=1$, and $k=0$, we have
\begin{align*}
	(\lambda^j \pd_{x'}^{\alpha'} \CC_2^1(\lambda)\eta)(x)
	&=
	\vp(x_N) \int^\infty_0 \CF^{-1}_{\xi'}\Big{[}\fm(\xi', \lambda)\frac{\lambda^j (i\xi')^{\alpha'} |\xi'|}{1+|\xi'|^2}\\
	&\hspace{80pt} \times 
e^{-|\xi'|(x_N+y_N)} \vp(y_N) \widehat{(1-\Delta') \eta} (\xi', y_N)\Big{]}(x')\,dy_N,\\
	(\lambda^j \pd_{x'}^{\alpha'} \CC_2^2(\lambda)\eta)(x)
	&=-
	\vp(x_N) \int^\infty_0 \CF^{-1}_{\xi'}\Big{[}\fm(\xi', \lambda)\frac{\lambda^j (i\xi')^{\alpha'}}
{(1+|\xi'|^2)|\xi'|} \\
		&\hspace{80pt} \times |\xi'| e^{-|\xi'|(x_N+y_N)} \pd_N(\vp(y_N) \hat{\eta} (\xi', y_N))\Big{]}(x')\,dy_N\\
	&\enskip+
	\sum^{N-1}_{l=1} \vp(x_N) \int^\infty_0 \CF^{-1}_{\xi'}\Big{[}\fm(\xi', \lambda)\frac{\lambda^j (i\xi')^{\alpha'}(i\xi_l)}{(1+|\xi'|^2)}\\
		&\hspace{80pt} \times 
	e^{-|\xi'|(x_N+y_N)} \pd_N(\vp(y_N) \widehat{\pd_l \eta} (\xi', y_N))\Big{]}(x')\,dy_N.
\end{align*}
Furthermore, for $j=0, 1$, $|\alpha'|=0$, and $k=1$, we have
\begin{align*}
	(\lambda^j \pd_N \CC_2^1(\lambda)\eta)(x)
	&=
	(\pd_N \vp)(x_N) \int^\infty_0 \CF^{-1}_{\xi'}\Big{[}\fm(\xi', \lambda)\frac{\lambda^j}{1+|\xi'|^2} \\
		&\hspace{80pt} \times 
	|\xi'| e^{-|\xi'|(x_N+y_N)} \vp(y_N) \widehat{(1-\Delta') \eta} (\xi', y_N)\Big{]}(x')\,dy_N\\
	&\enskip-
	\vp(x_N) \int^\infty_0 \CF^{-1}_{\xi'}\Big{[}\fm(\xi', \lambda)\frac{\lambda^j |\xi'|^2}{1+|\xi'|^2}		
	e^{-|\xi'|(x_N+y_N)} \vp(y_N) \widehat{(1-\Delta') \eta} (\xi', y_N)\Big{]}(x')\,dy_N,\\
	(\lambda^j \pd_N \CC_2^2(\lambda)\eta)(x)
	&=-
	(\pd_N \vp)(x_N) \int^\infty_0 \CF^{-1}_{\xi'}\Big{[}\fm(\xi', \lambda)\frac{\lambda^j}{1+|\xi'|^2}
	e^{-|\xi'|(x_N+y_N)} \pd_N(\vp(y_N) \hat{\eta} (\xi', y_N))\Big{]}(x')\,dy_N\\
	&\enskip +\sum^{N-1}_{l=1} 
	(\pd_N \vp)(x_N) \int^\infty_0 \CF^{-1}_{\xi'}\Big{[}\fm(\xi', \lambda)\frac{\lambda^j(i\xi_l)}{(1+|\xi'|^2)|\xi'|}\\
		&\hspace{80pt} \times |\xi'| e^{-|\xi'|(x_N+y_N)} \pd_N(\vp(y_N) \widehat{\pd_l \eta} (\xi', y_N))\Big{]}(x')\,dy_N\\
	&\enskip+
	\vp(x_N) \int^\infty_0 \CF^{-1}_{\xi'}\Big{[}\fm(\xi', \lambda)\frac{\lambda^j}{1+|\xi'|^2}
	|\xi'| e^{-|\xi'|(x_N+y_N)} \pd_N(\vp(y_N) \hat{\eta} (\xi', y_N))](x')\,dy_N\\
	&\enskip-
	\sum^{N-1}_{l=1} \vp(x_N) \int^\infty_0 \CF^{-1}_{\xi'}\Big{[}\fm(\xi', \lambda)\frac{\lambda^j |\xi'|(i\xi_l)}{1+|\xi'|^2}
	e^{-|\xi'|(x_N+y_N)} \pd_N(\vp(y_N) \widehat{\pd_l \eta} (\xi', y_N))\Big{]}(x')\,dy_N.
\end{align*}
Then by Corollary \ref{fm} and Leibiz's rule, 
there exists $\lambda_0 \ge 1$ such that 
\begin{align*}
	&
	\fm(\xi', \lambda)\frac{\lambda^j (i\xi')^{\alpha'}|\xi'|}{1+|\xi'|^2}, \enskip 
	\fm(\xi', \lambda)\frac{\lambda^j (i\xi')^{\alpha'}}{(1+|\xi'|^2)|\xi'|},\enskip 
	\fm(\xi', \lambda)\frac{\lambda^j (i\xi')^{\alpha'} (i\xi_l)}{1+|\xi'|^2},\enskip
	\fm(\xi', \lambda)\frac{\lambda^j}{1+|\xi'|^2}, \\
	&\fm(\xi', \lambda)\frac{\lambda^j |\xi'|^2}{1+|\xi'|^2}, \enskip
	\fm(\xi', \lambda)\frac{\lambda^j (i\xi_l)}{(1+|\xi'|^2)|\xi'|}, \enskip 
	\fm(\xi', \lambda)\frac{\lambda^j |\xi'|(i\xi_l)}{1+|\xi'|^2}
\end{align*}
are class of $\bM^2_0(\Sigma_{\epsilon, \lambda_0})$ for $j=0, 1$ and $l=1, \ldots N-1$.
On the other hand, for $j=0, 1$ and $2\le |\alpha'|+k \le 3-j$, we have
\begin{align*}
	(\lambda^j \pd_{x'}^{\alpha'} \pd_N^k \CC_2^1(\lambda)\eta)(x)
	&=\sum_{k_1+k_2=k}
	\dbinom{k}{k_1}
	(\pd_N^{k_1} \vp)(x_N) \int^\infty_0 \CF^{-1}_{\xi'}\Big{[}\fm(\xi', \lambda)\frac{\lambda^j (i\xi')^{\alpha'} (-|\xi'|)^{k_2}}{1+|\xi'|^2}\\
	&\hspace{80pt} \times |\xi'| e^{-|\xi'|(x_N+y_N)} \vp(y_N) \widehat{(1-\Delta') \eta} (\xi', y_N)\Big{]}(x')\,dy_N,\\
	(\lambda^j \pd_{x'}^{\alpha'} \pd_N^k \CC_2^2(\lambda)\eta)(x)
	&=-\sum_{k_1+k_2=k}
	\dbinom{k}{k_1}
	(\pd_N^{k_1} \vp)(x_N) \int^\infty_0 \CF^{-1}_{\xi'}\Big{[}\fm(\xi', \lambda)\frac{\lambda^j (i\xi')^{\alpha'} (-|\xi'|)^{k_2}}{1+|\xi'|^2}\\
	&\hspace{80pt} \times e^{-|\xi'|(x_N+y_N)} \pd_N(\vp(y_N) \hat{\eta} (\xi', y_N))\Big{]}(x')\,dy_N\\
	&\enskip +\sum_{l=1}^{N-1} \sum_{k_1+k_2=k}
	\dbinom{k}{k_1}
	(\pd_N^{k_1} \vp)(x_N) \int^\infty_0 \CF^{-1}_{\xi'}\Big{[}\fm(\xi', \lambda)\frac{\lambda^j (i\xi')^{\alpha'} (-|\xi'|)^{k_2}(i\xi_l)}{(1+|\xi'|^2)|\xi'|}\\
	&\hspace{80pt} \times |\xi'| e^{-|\xi'|(x_N+y_N)} \pd_N(\vp(y_N) \widehat{\pd_l \eta} (\xi', y_N))\Big{]}(x')\,dy_N,
\end{align*}
then 
Corollary \ref{fm} and Leibiz's rule implies that
there exists $\lambda_0 \ge 1$ satisfying 
\[
	\fm(\xi', \lambda)\frac{\lambda^j (i\xi')^{\alpha'} (-|\xi'|)^{k_2}}{1+|\xi'|^2},
	\enskip \fm(\xi', \lambda)\frac{\lambda^j (i\xi')^{\alpha'} (-|\xi'|)^{k_2}(i\xi_l)}{(1+|\xi'|^2)|\xi'|} \in \bM^2_0(\Sigma_{\epsilon, \lambda_0}).
\]
Therefore Lemma \ref{lem:t3} gives us
there exists a positive constant $r$ such that
\begin{equation} \label{rbddh2}
\begin{aligned}
&\CR_{\CL(H^2_q(\R^N_+), L_q(\R^N_+))}
(\{(\tau \pd_\tau)^n \lambda^j \pd_{x'}^{\alpha'} \pd_N^k \CC_2^a (\lambda) \mid 
\lambda \in \Sigma_{\epsilon, \lambda_0}\}) 
\leq r
\end{aligned}
\end{equation}
for $n = 0, 1$, $j=0, 1$, $1\le |\alpha'|+k \le 3-j$ and $a=1, 2$.
Combining \eqref{rbddh1} and \eqref{rbddh2}, the operator $\CC_2(\lambda)$ satisfies \eqref{est:rbdd}.
This completes the proof of Proposition \ref{with}.
\section{The result for the case $\gamma_* \in \R$}\label{general}
Assume that $\gamma_* \in \R$ and $\mu$, $\nu$, $\kappa$, and $\sigma$ satisfy the same condition \eqref{condi} as in Theorem \ref{main}.
We recall that the following resolvent problem.
\begin{equation}\label{r9}
\left\{
\begin{aligned}
	&\lambda \rho + \rho_* \dv \bu = d & \quad&\text{in $\R^N_+$}, \\
	&\rho_* \lambda \bu - \DV\{\bS(\bu) - (\gamma_* - \rho_* \kappa \Delta) \rho \bI\}=\bff& \quad&\text{in $\R^N_+$},\\
	&\{\bS(\bu) - (\gamma_* - \rho_* \kappa \Delta) \rho \bI\} \bn - \sigma \Delta' h \bn =\bg & \quad&\text{on $\R^N_0$},\\
	&\bn \cdot \nabla \rho = k & \quad&\text{on $\R^N_0$},\\
	&\lambda h - \bu \cdot \bn = \zeta & \quad&\text{on $\R^N_0$}.
\end{aligned}
\right.
\end{equation}
Note that \eqref{r9} is equivalent to \eqref{main prob} if $\gamma_* = 0$.
Using Theorem \ref{main}, we also obtain the existence of $\CR$-bounded operator families for \eqref{r9} as follows. 
\begin{theorem}\label{thm:general}
Let $1<q<\infty$.
Assume that $\mu, \nu, \kappa, \sigma >0$, and $\gamma_* \in \R$
are constants satisfying \eqref{condi}.
Let $\epsilon \in (\epsilon_*, \pi/2)$ for $\epsilon_*$ given in Theorem \ref{main}. 
Then there exists a constant $\lambda_0 \ge 1$ such that the following assertions hold true: 

$\thetag1$ 
For any $\lambda \in \Sigma_{\epsilon, \lambda_0}$ there exist operator families 
\begin{align*}
&\CA_{\gamma_{*}} (\lambda) \in 
{\rm Hol} (\Sigma_{\epsilon, \lambda_0}, 
\CL(\CX_q(\R^N_+), H^3_q(\R^N_+)))\\
&\CB_{\gamma_{*}} (\lambda) \in 
{\rm Hol} (\Sigma_{\epsilon, \lambda_0}, 
\CL(\CX_q(\R^N_+), H^2_q(\R^N_+)^N)),\\
&\CC_{\gamma_{*}} (\lambda) \in 
{\rm Hol} (\Sigma_{\epsilon, \lambda_0}, 
\CL(\CX_q(\R^N_+), W^{3-1/q}_q(\R^N_0)))
\end{align*}
such that 
for any $\bF=(d, \bff, \bg, k, \zeta) \in X_q(\R^N_+)$, 
\begin{equation*}
\rho = \CA_{\gamma_{*}} (\lambda) \CF_\lambda \bF, \quad
\bu = \CB_{\gamma_{*}} (\lambda) \CF_\lambda \bF, \quad
h = \CC_{\gamma_{*}} (\lambda) \CF_\lambda \bF
\end{equation*}
are unique solutions of problem \eqref{r9}.

$\thetag2$ 
There exists a positive constant $r$ such that
\begin{equation} \label{rbdd}
\begin{aligned}
&\CR_{\CL(\CX_q(\R^N_+), \fA_q(\R^N_+))}
(\{(\tau \pd_\tau)^n \CR_\lambda \CA_{\gamma_{*}} (\lambda) \mid 
\lambda \in \Sigma_{\epsilon, \lambda_0}\}) 
\leq r,\\
&\CR_{\CL(\CX_q(\R^N_+), \fB_q(\R^N_+))}
(\{(\tau \pd_\tau)^n \CS_\lambda \CB_{\gamma_{*}} (\lambda) \mid 
\lambda \in \Sigma_{\epsilon, \lambda_0}\}) 
\leq r,\\
&\CR_{\CL(\CX_q(\R^N_+), \fC_q(\R^N_+))}
(\{(\tau \pd_\tau)^n \CT_\lambda \CC_{\gamma_{*}} (\lambda) \mid 
\lambda \in \Sigma_{\epsilon, \lambda_0}\}) 
\leq r
\end{aligned}
\end{equation}
for $n = 0, 1$.
Here, above constants $\lambda_0$ and $r$ depend solely on $N$, $q$, $\epsilon$,
$\mu$, $\nu$, $\kappa$, $\gamma_*$ and $\sigma$. 
\end{theorem}

\begin{proof}
For $\bF=(d, \bff, \bg, k, \eta) \in X_q (\R^N_+)$,
Theorem \ref{main} implies that 
$\rho = \CA_0(\lambda) \CF_\lambda \bF$, 
$\bu = \CB_0(\lambda) \CF_\lambda \bF$, and $h = \CC_0(\lambda) \CF_\lambda \bF$
satisfy
\begin{equation*}\label{r3}
\left\{
\begin{aligned}
	&\lambda \rho + \dv \bu = d & \quad&\text{in $\R^N_+$}, \\
	&\lambda \bu - \DV\{\bS(\bu) - (\gamma_* - \kappa \Delta)  \rho\bI\}=\bff + \gamma_* \nabla \CA_0(\lambda) \CF_\lambda \bF& \quad&\text{in $\R^N_+$},\\
	&\{\bS(\bu) - (\gamma_* - \rho_* \kappa \Delta) \rho \bI\} \bn - \sigma \Delta' h \bn=\bg - \gamma_* (\CA_0(\lambda) \CF_\lambda \bF) \bn& \quad&\text{on $\R^N_0$},\\
	&\bn \cdot \nabla \rho = k & \quad&\text{on $\R^N_0$},\\
	&\lambda h - \bu \cdot \bn = \zeta & \quad&\text{on $\R^N_0$}.
\end{aligned}
\right.\end{equation*}
Denote 
\begin{align*}
	\CD(\lambda) \bF &=(0, -\gamma_* \nabla \CA_0(\lambda) \CF_\lambda \bF, \gamma_* (\CA_0(\lambda) \CF_\lambda \bF) \bn, 0, 0),\\
\end{align*}
Employing the same method as in \cite[Sec.6]{S}, we see that the existence of the inverse operator 
$(I-\CD(\lambda))^{-1} \in \CL(X_q(\R^N_+))$, and then we can construct the solution to \eqref{r9}:
\[
	\rho = \CA_0(\lambda) \CF_\lambda (I-\CD(\lambda))^{-1}\bF, \quad
	\bu = \CB_0(\lambda) \CF_\lambda (I-\CD(\lambda))^{-1}\bF, \quad
	h = \CC_0(\lambda) \CF_\lambda (I-\CD(\lambda))^{-1}\bF.
\]
Furthermore, setting $\CA_{\gamma_{*}}(\lambda) \CF_\lambda \bF=\CA_0(\lambda) \CF_\lambda (I-\CD(\lambda))^{-1}\bF$,
$\CB_{\gamma_{*}}(\lambda) \CF_\lambda \bF=\CB_0(\lambda) \CF_\lambda (I-\CD(\lambda))^{-1}\bF$,
$\CC_{\gamma_{*}}(\lambda) \CF_\lambda \bF=\CC_0(\lambda) \CF_\lambda (I-\CD(\lambda))^{-1}\bF$ for $\bF \in X_q(\R^N_+)$,
the operators $\CA_{\gamma_{*}}(\lambda)$, $\CB_{\gamma_{*}}(\lambda)$, and $\CC_{\gamma_{*}}(\lambda)$ satisfy \eqref{rbdd} 
by \eqref{rbdd1} and Lemma \ref{lem:5.3}.
The uniqueness of solutions to \eqref{r9} also follow from \eqref{rbdd1} with $n=0$ by using the same manner as in \cite[Sec.6]{S}.
This completes the proof of Theorem \ref{thm:general}.

\end{proof}

\section{Maximal $L_p$-$L_q$ regularity}\label{mr}
In this section, we consider the following linearized problems:
\begin{equation}\label{linear1}
\left\{
\begin{aligned}
	&\pd_t \rho + \rho_* \dv \bu = 0 & \quad&\text{in $\R^N_+ \times (0, \infty)$}, \\
	&\rho_* \pd_t  \bu - \DV\{\bS(\bu) - (\gamma_* - \rho_* \kappa \Delta) \rho \bI\}=0& \quad&\text{in $\R^N_+ \times (0, \infty)$},\\
	&\{\bS(\bu) - (\gamma_* - \rho_* \kappa \Delta) \rho \bI\} \bn - \sigma \Delta' h \bn =0 & \quad&\text{on $\R^N_0 \times (0, \infty)$},\\
	&\bn \cdot \nabla \rho = 0 & \quad&\text{on $\R^N_0 \times (0, \infty)$},\\
	&\pd_t  h - \bu \cdot \bn = 0 & \quad&\text{on $\R^N_0 \times (0, \infty)$},\\
	&(\rho, \bu, h)|_{t=0} = (\rho_0, \bu_0, h_0)& \quad&\text{in $\R^N_+$},
\end{aligned}
\right.
\end{equation}
\begin{equation}\label{linear2}
\left\{
\begin{aligned}
	&\pd_t \rho + \rho_* \dv \bu = d & \quad&\text{in $\R^N_+ \times (0, \infty)$}, \\
	&\rho_* \pd_t  \bu - \DV\{\bS(\bu) - (\gamma_* - \rho_* \kappa \Delta) \rho \bI\}=\bff& \quad&\text{in $\R^N_+ \times (0, \infty)$},\\
	&\{\bS(\bu) - (\gamma_* - \rho_* \kappa \Delta) \rho \bI\} \bn - \sigma \Delta' h \bn =\bg & \quad&\text{on $\R^N_0 \times (0, \infty)$},\\
	&\bn \cdot \nabla \rho = k & \quad&\text{on $\R^N_0 \times (0, \infty)$},\\
	&\pd_t  h - \bu \cdot \bn = \zeta & \quad&\text{on $\R^N_0 \times (0, \infty)$},\\
	&(\rho, \bu, h)|_{t=0} = (0, 0, 0)& \quad&\text{in $\R^N_+$}.
\end{aligned}
\right.
\end{equation}
First, we consider \eqref{linear1} in the semigroup setting.
Let
\[
	\X_q(\R^N_+) = H^1_q(\R^N_+) \times L_q(\R^N_+)^N \times W^{2-1/q}_q(\R^N_0)
\]
with the norm
\[
	\|(\rho, \bu, h)\|_{\X_q(\R^N_+)}=\|\rho\|_{H^1_q(\R^N_+)} + \|\bu\|_{L_q(\R^N_+)} + \|h\|_{W^{2-1/q}_q(\R^N_0)}
\]
and let $\CA$, $\CD(\CA)$, and $\|\cdot\|_{\CD(\CA)}$ be an operator, its domain, and the norm with
\begin{align*}
	\CD_q(\CA)&=\{(\rho, \bu, h) \in (H^3_q(\R^N_+) \times H^2_q(\R^N_+)^N \times W^{3-1/q}_q(\R^N_0) \\
	&\qquad \mid \{\bS(\bu) - (\gamma_* - \rho_* \kappa \Delta) \rho \bI\} \bn - \sigma \Delta' h \bn =0, \enskip
	\bn \cdot \nabla \rho = 0 \enskip \text{on $\R^N_0$} \},\\
	\|(\rho, \bu, h)\|_{\CD_q(\CA)}&=\|\rho\|_{H^3_q(\R^N_+)} + \|\bu\|_{H^2_q(\R^N_+)} + \|h\|_{W^{3-1/q}_q(\R^N_0)},\\
	\CA(\rho, \bu, h)&=(\rho_* \dv \bu, \enskip \rho_*^{-1} \DV\{\bS(\bu) - (\gamma_* - \rho_* \kappa \Delta) \rho \bI\}, \enskip
	\bu \cdot \bn) \enskip \text{for $(\rho, \bu, h) \in \CD_q(\CA)$}. 
\end{align*}
Then setting $\bU = (\rho, \bu, h)$ and $\bU_0=(\rho_0, \bu_0, h_0)$, \eqref{linear1} can be written by 
\[
	\pd_t \bU-\CA \bU=0 \enskip \text{in $\R^N_+ \times \R^N_0$, for $t \in (0, \infty)$},
	\quad
	\bU|_{t=0}=\bU_0.
\]
Theorem \ref{thm:general} implies that the resolvent set $\rho(\CA)$ contains $\Sigma_{\epsilon, \lambda_0}$.
Furthermore, since Definition \ref{dfn2} with $n = 1$ 
implies the uniform boundedness
of the operator family $\CT$, 
the solution $\bU$ of the following equation
\[
	\lambda \bU-\CA \bU=\bF \enskip \text{in $\R^N_+ \times \R^N_0$} 
\]
satisfies 
the resolvent estimate:
\begin{equation}\label{resolvent es}
	|\lambda| \|\bU\|_{\X_q(\R^N_+)} + \|\bU\|_{\CD_q(\R^N_+)} \leq C_r
	\|\bF\|_{\X_q(\R^N_+)}
\end{equation}
for any $\lambda \in \Sigma_{\epsilon, \lambda_0}$ and $\bF \in \X_q(\R^N_+)$.
Thus
the generation of an analytic semigroup follows from standard analytic semigroup arguments.

\begin{theorem}\label{thm:semi1}
Let $1 < q < \infty$.
Assume that $\mu, \nu, \kappa, \sigma >0$, and $\gamma_* \in \R$
are constants satisfying \eqref{condi}.
Then, the operator $\CA$ generates an analytic
semigroup $\{T(t)\}_{t\geq 0}$ on $\X_q(\R^N_+)$.  
Moreover, there exist
constants $\gamma_0 \geq 1$ and $C_{q, N, \gamma_0} > 0$
such that $\{T(t)\}_{t\geq 0}$ satisfies the estimates: 
\begin{align*}
	\|T(t) (\rho_0, \bu_0, h_0) \|_{\X_q(\R^N_+)}
	&\leq C_{q, N, \gamma_0} e^{\gamma_0 t} \|(\rho_0, \bu_0, h_0)\|_{\X_q(\R^N_+)},\\
	\|\pd_t T(t) (\rho_0, \bu_0, h_0) \|_{\X_q(\R^N_+)}
	&\leq C_{q, N, \gamma_0} e^{\gamma_0 t} t^{-1} \|(\rho_0, \bu_0, h_0)\|_{\X_q(\R^N_+)},\\
	\|\pd_t T(t) (\rho_0, \bu_0, h_0) \|_{\X_q(\R^N_+)}
	&\leq C_{q, N, \gamma_0} e^{\gamma_0 t} \|(\rho_0, \bu_0, h_0)\|_{\CD_q(\CA)}
\end{align*}
for any $t > 0$.
\end{theorem}

Now we state the maximal $L_p$-$L_q$ regularity estimates for \eqref{linear1}.
Let 
\[
	\CD_{q, p}(\R^N_+) = (\X_q(\R^N_+), \CD_q(\CA))_{1-1/p, p}
\]
with $\|(\rho, \bu, h)\|_{\CD_{q, p} (\R^N_+)}=\|\rho\|_{B^{3-2/p}_{q, p}(\R^N_+)}+\|\bu\|_{B^{2(1-1/p)}_{q, p}(\R^N_+)}+\|h\|_{B^{3-1/q-1/p}_{q, p}(\R^N_0)}$.
Combining Theorem \ref{thm:semi1} with a real interpolation method
(cf. Shibata and Shimizu \cite[Proof of Theorem 3.9]{SS2}), we have
the following result. 

\begin{theorem}\label{thm:semi2}
Let $1 < p, q < \infty$. 
Assume that $\mu, \nu, \kappa, \sigma >0$, and $\gamma_* \in \R$
are constants satisfying \eqref{condi}.
Then for any $(\rho_0, \bu_0, h_0) \in \CD_{q, p} (\R^N_+)$, 
\eqref{linear1}
admits a unique solution $(\rho, \bu, h) = T(t) (\rho_0, \bu_0, h_0)$
possessing the estimate:
\begin{equation}\label{semi2}
\begin{aligned}
	&\|e^{-\gamma t} \pd_t (\rho, \bu, h)\|_{L_p(\R_+, \X_q(\R^N_+))}
	+ \|e^{-\gamma t} (\rho, \bu, h)\|_{L_p(\R_+, \CD_q(\R^N_+))}\\
	&\leq C_{p, q, N, \gamma_0} 
	\|(\rho_0, \bu_0, h_0)\|_{\CD_{q, p} (\R^N_+)}
\end{aligned}
\end{equation}
for any $\gamma \geq \gamma_0$. 
\end{theorem}

Second, we consider \eqref{linear2}.
To prove the maximal $L_p$-$L_q$ regularity for \eqref{linear2}, 
the key tool is the Weis's
operator valued Fourier multiplier theorem.
Let $\CD(\BR,X)$ and $\CS(\BR,X)$ be the set of all $X$ 
valued $C^{\infty}$ functions having compact support 
and the Schwartz space of rapidly decreasing $X$ 
valued functions, respectively,
while $\CS'(\BR,X)= \CL(\CS(\BR),X)$. 
Given $M \in L_{1,\rm{loc}}(\BR \backslash \{0\}, \CL(X, Y))$, 
we define the operator $T_{M} : \CF^{-1} \CD(\BR,X)\rightarrow \CS'(\BR,Y)$ 
by
\begin{align}\label{eqTM}
T_M \phi=\CF^{-1}[M\CF[\phi]],\quad (\CF[\phi] \in \CD(\BR,X)). 
\end{align}

\begin{theorem}[Weis \cite{W}]\label{Weis}
Let $X$ and $Y$ be two UMD Banach spaces and $1 < p < \infty$. 
Let $M$ be a function in $C^{1}(\BR \backslash \{0\}, \CL(X,Y))$ such that 
\begin{align*}
	&\CR_{\CL(X,Y)} (\{M(\tau) \mid
 	\tau \in \BR \backslash \{0\}\}) \leq \beta_1 < \infty,\\
	&\CR_{\CL(X,Y)} (\{\tau M'(\tau) \mid
 	\tau \in \BR \backslash \{0\}\}) \leq \beta_2 < \infty
\end{align*}
with some constants $\beta_1$ and $\beta_2$. 
Then, the operator $T_{M}$ defined in \eqref{eqTM} 
is extended to a bounded linear operator from
 $L_{p}(\BR,X)$ into $L_{p}(\BR,Y)$. 
Moreover, denoting this extension by $T_{M}$, we have 
\begin{align*}
\|T_{M}\|_{\CL(L_p(\BR,X),L_p(\BR,Y))} \leq C(\beta_1+\beta_2)
\end{align*}
for some positive constant $C$ depending on $p$, $X$ and $Y$. 
\end{theorem}

To state the maximal $L_p$-$L_q$ regularity for \eqref{linear2},
we introduce the following notation:
\begin{align*}
	\CF_{p, q, \gamma}
	=\{(d, \bff, \bg, k, \zeta) \mid &d \in L_{p, \gamma, 0}(\R, H^1_q(\R^N_+)), \enskip
	\bff \in L_{p, \gamma, 0}(\R, L_q(\R^N_+)^N),\\
	&\bg \in L_{p, \gamma, 0}(\R, H^1_q(\R^N_+)^N) \cap H^{1/2}_{p, \gamma, 0}(\R, L_q(\R^N_+)^N), \\
	&k \in H^1_{p, \gamma, 0}(\R, L_q(\R^N_+)) \cap L_{p, \gamma, 0}(\R, H^2_q(\R^N_+)), \\
	&\zeta \in L_{p, \gamma, 0}(\R, W^{2-1/q}_q(\R^N_0))\}
\end{align*}
with the norm
\begin{align*}
	&\|(d, \bff, \bg, k, \zeta)\|_{\CF_{p, q, \gamma}}\\
	&\enskip=\|e^{-\gamma t} d \|_{L_p(\R, H^1_q(\R^N_+))} 
	+ \|e^{-\gamma t} (\bff, \Lambda^{1/2}_\gamma \bg) \|_{L_p(\R, L_q(\R^N_+))}
	+ \|e^{-\gamma t} \bg \|_{L_p(\R, H^1_q(\R^N_+))}\\
	&\enskip+ \|e^{-\gamma t} (\pd_t k, \Lambda^{1/2}\nabla k)\|_{L_p(\R, L_q(\R^N_0))}
	+ \|e^{-\gamma t} k \|_{L_p(\R, H^2_q(\R^N_0))}
	+ \|e^{-\gamma t} \zeta \|_{L_p(\R, W^{2-1/q}_q(\R^N_0))}.
\end{align*}
\begin{theorem}\label{thm:mr} 
Let $1 < p, q < \infty$.
Assume that $\mu, \nu, \kappa, \sigma >0$, and $\gamma_* \in \R$
are constants satisfying \eqref{condi}. 
Then
there exists a constant $\gamma_1 \ge 1$ such that for any $(d, \bff, \bg, k, \zeta) \in \CF_{p, q, \gamma}$,
\eqref{linear2} admits a unique solution 
$(\rho, \bu, h)$ with 
\begin{align*}
	&\rho \in L_{p, \gamma_1, 0}(\R, H^3_q(\R^N_+)) \cap H^1_{p, \gamma_1, 0}(\R, H^1_q(\R^N_+)), \\
	&\bu \in L_{p, \gamma_1, 0}(\R, H^2_q(\R^N_+)^N) \cap H^1_{p, \gamma_1, 0}(\R, L_q(\R^N_+)^N),\\
	&h \in L_{p, \gamma_1, 0}(\R, W^{3-1/q}_q(\R^N_0)) \cap H^1_{p, \gamma_1, 0}(\R, W^{2-1/q}_q(\R^N_0)),
\end{align*}
possessing the estimate
\begin{equation}\label{est:mr}
\begin{aligned}
	&\|e^{-\gamma t} \pd_t \rho\|_{L_p(\R, H^1_q(\R^N_+))} 
	+ \sum_{j=0}^3 \|e^{-\gamma t} \Lambda^{j/2} \rho\|_{L_p(\R, H^{3-j}_q(\R^N_+))}\\
	&+ \|e^{-\gamma t} \pd_t \bu\|_{L_p(\R, L_q(\R^N_+))}
	+ \sum_{j=0}^2 \|e^{-\gamma t} \Lambda^{j/2} \bu\|_{L_p(\R, H^{2-j}_q(\R^N_+))}\\
	&+ \|e^{-\gamma t} \pd_t h\|_{L_p(\R, W^{2-1/q}_q(\R^N_0))}
	+ \|e^{-\gamma t} h\|_{L_p(\R, W^{3-1/q}_q(\R^N_0))}\\
	&\leq C\|(d, \bff, \bg, k, \zeta)\|_{\CF_{p, q, \gamma}}
\end{aligned}
\end{equation}
for any $\gamma \ge \gamma_1$ with some constant $C$ depending 
on $N$, $p$, $q$, and $\gamma_1$. 
\end{theorem}

\begin{proof}
Let $\gamma_1 > \max\{\lambda_0, \gamma_0\}$, where $\lambda_0$ and $\gamma_0$ are
the same constants as in Theorem \ref{thm:general} and Theorem \ref{thm:semi2}, respectively.
Let $\bF=(d, \bff, \bg, k, \zeta) \in \CF_{p, q, \gamma}$ for $\gamma \ge \gamma_1$.
Applying Laplace transform to \eqref{linear2} yields
\begin{equation}\label{resolvent}
\left\{
\begin{aligned}
	&\lambda \hat \rho + \rho_* \dv \hat \bu = \hat d & \quad&\text{in $\R^N_+$}, \\
	&\rho_* \lambda \hat \bu - \DV\{\bS(\hat \bu) - (\gamma_* - \rho_* \kappa \Delta) \hat \rho \bI\}=\hat \bff& \quad&\text{in $\R^N_+$},\\
	&\{\bS(\hat \bu) - (\gamma_* - \rho_* \kappa \Delta) \hat \rho \bI\} \bn - \sigma \Delta' \hat h \bn =\hat \bg & \quad&\text{on $\R^N_0$},\\
	&\bn \cdot \nabla \hat \rho = \hat k & \quad&\text{on $\R^N_0$},\\
	&\lambda \hat h - \hat \bu \cdot \bn = \hat \zeta & \quad&\text{on $\R^N_0$},\\
\end{aligned}
\right.
\end{equation}
where we have set $\hat f = \CL[f]$.
Let $\bF_\lambda=(d, \bff, \nabla \bg, \Lambda^{1/2} \bg, \nabla^2 k, \Lambda^{1/2} \nabla k, \pd_t k, \zeta)$.
Theorem \ref{thm:general} yields that
\begin{equation}\label{sol of resolvent}
	\CL[\bU] = (\hat \rho, \hat \bu, \hat h) = (\CA_{\gamma_*}(\lambda) \widehat \bF_\lambda,
	\CB_{\gamma_*}(\lambda) \widehat \bF_\lambda,
	\CC_{\gamma_*}(\lambda) \widehat \bF_\lambda)
\end{equation}
is a solution of \eqref{resolvent} for $\lambda = \gamma + i\tau \in \C$ with $\gamma \ge \gamma_1$ and $\tau \in \R$.
Here we note that
since $\bF_\lambda=0$ if $t<0$, $\widehat \bF_\lambda$ is holomorphic for $\lambda$ if $\gamma \ge \gamma_1$,
thus the representation
\eqref{sol of resolvent} is independent of $\gamma$ by Cauchy's theorem.
Therefore the solution of \eqref{linear2} has the form:
\begin{equation}\label{sol of linear2}
	\bU=(\rho, \bu, h) = (\CL^{-1}[\CA_{\gamma_*}(\lambda) \widehat \bF_\lambda],
	\CL^{-1}[\CB_{\gamma_*}(\lambda) \widehat \bF_\lambda],
	\CL^{-1}[\CC_{\gamma_*}(\lambda) \widehat \bF_\lambda]).
\end{equation}
Note that the Laplace transform $\CL$ and the Laplace inverse transform $\CL^{-1}$ are written by Fourier transform
$\CF$ and Fourier inverse transform $\CF^{-1}$ in $\R$ as 
\[
\CL[f](\lambda) = \CF[e^{-\gamma t}f(t)](\tau),
\quad \CL^{-1}[g](t) = e^{\gamma t}\CF^{-1}[g(\tau)](t),
\]
where $\lambda = \gamma+i\tau \in \C$. Thus \eqref{sol of linear2} is rewritten as
\[
	e^{-\gamma t}\bU = (\CF^{-1}[\CA_{\gamma_*}(\lambda) \CF[e^{-\gamma t} \bF_\lambda]],
	\CF^{-1}[\CB_{\gamma_*}(\lambda) \CF[e^{-\gamma t} \bF_\lambda]],
	\CF^{-1}[\CC_{\gamma_*}(\lambda) \CF[e^{-\gamma t} \bF_\lambda]]).
\]
By Theorem \ref{thm:general} and Theorem \ref{Weis}, 
we have
\eqref{est:mr} for any $\gamma \ge \gamma_1$.

Next, we verify $(\rho, \bu, h)=(0, 0, 0)$ if $t<0$.
Note that
\[
	\gamma \bU = (\CL^{-1}[(\gamma/\lambda) \lambda \CA_{\gamma_*}(\lambda) \widehat \bF_\lambda],
	\CL^{-1}[(\gamma/\lambda) \lambda \CB_{\gamma_*}(\lambda) \widehat \bF_\lambda],
	\CL^{-1}[(\gamma/\lambda) \lambda \CC_{\gamma_*}(\lambda) \widehat \bF_\lambda]).
\]
Since $|(\tau \pd_\tau)^\ell\gamma/\lambda| \leq 1$ for $\lambda
\in \Sigma_{\epsilon, \lambda_0}$, we have
\begin{align*}
	\gamma \|\bU\|_{L_p((-\infty, 0), \X_q(\R^N_+))}
	&\leq \gamma \|e^{-\gamma t} \bU\|_{L_p(\R, \X_q(\R^N_+))}
	\leq \gamma \|e^{-\gamma_1 t} \bU\|_{L_p(\R, \X_q(\R^N_+))}\\
	&\leq C\|\bF\|_{\CF_{p, q, \gamma_1}}
\end{align*}
for any $\gamma \ge \gamma_1$, which follows from \cite[Proposition 3.6]{DHP}. 
Thus letting $\gamma \to \infty$
yields that $\|\bU\|_{L_p((-\infty, 0), \X_q(\R^N_+))}=0$.
By trace theorem in
theory of real interpolation, $\bU \in C(\R, B^{3-2/p}_{q, p}(\R^N_+) \times B^{2(1-1/p)}_{q, p}(\R^N_+) \times B^{3-1/q-1/p}_{q, p}(\R^N_0))$, therefore 
we have $\bU(\cdot, t)=0$ for $t < 0$, which implies that $\bU$ satisfies the initial condition.

Finally, we prove the uniqueness of solutions.  
Let us consider the homogeneous equation:
\begin{equation}\label{homo U}
	\pd_t \bU-\CA \bU=0 \enskip \text{in $\R^N_+ \times \R^N_0$, for $t \in (0, \infty)$},
	\quad
	\bU|_{t=0}=0.
\end{equation}
By Theorem \ref{thm:semi2}, we know that \eqref{homo U} has a solution $\bU$ with
\begin{equation}\label{space U}
	e^{-\gamma_0 t} \bU \in H^1_p((0, \infty), \X_q(\R^N_+)) \cap L_p((0, \infty), \CD_q(\CA)).
\end{equation}
Let $\bV$ be the zero extension of $\bU$ to $t<0$.
Then \eqref{homo U} implies that $\bV$ satisfies  
\[
	\pd_t \bV-\CA \bV=0 \enskip \text{in $\R^N_+ \times \R^N_0$, for $t \in \R$}.
\]
For any $\lambda \in \C$ with $\Re \lambda=\gamma \ge \gamma_1$, 
we set 
\[
	\widehat \bU(\lambda) = \int^\infty_{-\infty} e^{-\lambda t} \bV(t)\,dt
	=\int^\infty_0 e^{-\lambda t} \bU(t)\,dt.
\]
H\"older inequality and \eqref{space U} implies that
\begin{align*}
	\|\widehat \bU(\lambda)\|_{\CD_q(\CA)}
	&\le \left( \int^\infty_0 e^{-(\gamma - \gamma_0)tp'}\,dt \right)^{1/p'}
	\|e^{-\gamma_0 t}\bU\|_{L_p((0, \infty), \CD_q(\CA))}\\
	&=\{(\gamma - \gamma_0)p'\}^{-1/p'} \|e^{-\gamma_0 t}\bU\|_{L_p((0, \infty), \CD_q(\CA))}.
\end{align*}
Since $\lambda \widehat U = \int^\infty_0 e^{-\lambda t}\pd_t \bU\,dt$,
we also have
\begin{align*}
	\|\lambda \widehat \bU(\lambda)\|_{\X_q(\R^N_+)}
	&\le \{(\gamma - \gamma_0)p'\}^{-1/p'} \|e^{-\gamma_0 t} \pd_t\bU\|_{L_p((0, \infty), \X_q(\R^N_+))}.
\end{align*}
Therefore $\widehat \bU \in \CD_q(\CA)$ satisfies the resolvent problem:
\begin{equation}\label{resolvent U}
	\lambda \widehat \bU -\CA \widehat \bU = 0 \enskip \text{in $\R^N_+ \times \R^N_0$}.
\end{equation}
Theorem \ref{thm:general} implies that \eqref{resolvent U} has a unique solution for $\lambda \in \Sigma_{\epsilon, \lambda_0}$, thus we have $\widehat \bU(\lambda)=0$
for any $\lambda \in \C$ with $\gamma \ge \gamma_1$.
Applying the Laplace inverse transform to $\widehat \bU(\lambda)=0$, we have
$\bV(t)=0$ for $t \in \R$. Therefore we have $\bU(t)=0$ for $t>0$, which shows the uniqueness of \eqref{homo U}. 
This completes the proof of Theorem \ref{thm:mr}.
\end{proof}

\section{Conclusion}\label{sec13}

In this article, we consider $\CR$-boundedness of the solution operator families of the generalized resolvent problem for Korteweg type with surface tension in the half-space case. This $\CR$-boundedness can be a potential tool to investigate the maximal $L_p$-$L_q$ regularity using Weis's operator valued multiplier theorem. For future research, it can be analyze the local well-posedness of the model problem based on the result of the paper.


\backmatter





\bmhead{Acknowledgments}

The first author thank to Institute of Research and Community services for the support through International Research Collaboration (IRC)'s scheme 2023 contract number 27.55/UN23.37/PT.01.03/II/2023. The second author is partially supported by JSPS Grants-in-Aid for Early-Career Scientists 21K13819 and Grants-in-Aid for Scientific Research (B) 22H01134.


\section*{Declarations}

\begin{itemize}
\item {\bf Conflict of interest} There is no conflict of interests in this article 
\item {\bf Ethics approval}  Not applicable
\item {\bf Authors' contributions} All authors have been personally and actively involved in substantial work leading to the paper, and will take public responsibility for its content. Moreover, all authors are contributed equally in this work.
\end{itemize}

\end{document}